\documentclass{amsart}

\usepackage[numbers]{natbib}
\usepackage{amsmath, amssymb,stmaryrd}% ,mathdots?
\usepackage{type1cm}
\usepackage{tensor}
\usepackage{enumerate}
\usepackage{graphicx}
\usepackage[czech,english]{babel}
\usepackage{mdwtab}
\usepackage[OT1]{fontenc}
\usepackage[all]{xy}
\newcommand{\sig}{\lambda}

\newcommand{\entails}{\vdash}

\newcommand{\bbbn}{\mathbb{N}}

\newcommand{\bbbr}{\mathbb{R}}

\usepackage{color}
\definecolor{shadecolor}{gray}{.85}%
\definecolor{tintedcolor}{gray}{.80}%
\usepackage{framed}%
\definecolor{mytintedcolor}{gray}{.95}%
\makeatletter
\newdimen\svparindent
\newcounter{tmpthm}
\newenvironment{maintheorem}[1]{\setcounter{tmpthm}{\value{theorem}}%
\setcounter{theorem}{#1}\addtocounter{theorem}{-1}%
\begin{theorem}}{\setcounter{theorem}{\value{tmpthm}}\end{theorem}}
\newenvironment{mytinted}{%
  \MakeFramed {\FrameRestore}}%
{\endMakeFramed}
{\endlist\end{mytinted}\egroup}
\makeatother

\newcommand{\Type}[2]{\tensor*[^{#1}_{\scriptscriptstyle\mathbf #2}]{\sf\textstyle Tp}{}}

\newtheorem{theorem}{Theorem}

\newtheorem{corollary}{Corollary}
\newtheorem{lemma}[theorem]{Lemma}

\theoremstyle{definition}
\newtheorem{definition}[theorem]{Definition}

\theoremstyle{remark}
\newtheorem{remark}[theorem]{Remark}

\newtheorem{problem}{Problem}
\newtheorem{conjecture}{Conjecture}
\begin{document}

\title[Modeling Limits in Hereditary Classes]{Modeling Limits in Hereditary Classes:\\ Reduction and Application to Trees}
\author{Jaroslav Ne{\v s}et{\v r}il}
\address{Jaroslav Ne{\v s}et{\v r}il\\
Computer Science Institute of Charles University (IUUK and ITI)\\
   Malostransk\' e n\' am.25, 11800 Praha 1, Czech Republic}
%\curraddr{}
\email{nesetril@iuuk.mff.cuni.cz}
\thanks{Supported by grant ERCCZ LL-1201 
and CE-ITI P202/12/G061, and by the European Associated Laboratory ``Structures in
Combinatorics'' (LEA STRUCO)}

%    author two information
\author{Patrice Ossona~de~Mendez}
\address{Patrice~Ossona~de~Mendez\\
Centre d'Analyse et de Math\'ematiques Sociales (CNRS, UMR 8557)\\
  190-198 avenue de France, 75013 Paris, France
	--- and ---
Computer Science Institute of Charles University (IUUK)\\
   Malostransk\' e n\' am.25, 11800 Praha 1, Czech Republic}
\email{pom@ehess.fr}
\thanks{Supported by grant ERCCZ LL-1201 and by the European Associated Laboratory ``Structures in
Combinatorics'' (LEA STRUCO)}

%    \date is required; it is the date received by the editor.
\date{\today}

\subjclass[2010]{Primary  03C13 (Finite structures), 03C98 (Applications of model theory), 05C99 (Graph theory),  06E15 (Stone spaces and related structures), Secondary 28C05 (Integration theory via linear functionals)}

 \keywords{Graph \and Relational structure \and Graph limits \and Structural limits \and Radon measures \and Stone space \and Model theory \and First-order logic \and Measurable graph}

\begin{abstract}
Limits of graphs were initiated recently  in the two extreme contexts of dense and bounded degree graphs. 
This led to elegant limiting structures
called graphons and graphings. These approach have been unified and generalized by authors in a more general setting using a combination of analytic tools and model theory to ${\rm FO}$-limits (and $X$-limits) and to the notion of modeling. The existence of modeling limits was established for sequences in a bounded degree class and, in addition, to the case of classes of trees with bounded height and of graphs with bounded tree depth.
These seemingly very special classes is in fact a key step in the development of limits for more general situations. The natural obstacle for the existence of modeling limit for a monotone class of graphs is the nowhere dense property and it has been conjectured that this is a sufficient condition. Extending earlier results we derive several general results which present a realistic approach to this conjecture. As an example we then prove that the class of all finite trees admits modeling limits. 
\end{abstract}

\maketitle
\section{Introduction}
The study of limiting properties of large graphs have recently received a great attention, mainly in two
directions: limits of graphs with bounded degrees \cite{Benjamini2001} and limit of dense graphs \cite{Lov'asz2006}. 
These developments are nicely documented in  the recent monograph of Lov\'asz \cite{LovaszBook}. 
Motivated by a possible unifying scheme for the study of structural limits, we introduced the notion
of Stone pairing and ${\rm FO}$-convergence \cite{CMUC,NPOM1arxiv}.
Precisely, we proposed an approach 
based on the {\em Stone pairing} $\langle \phi,G\rangle$ of a first-order formula $\phi$ 
(with set of free variables ${\rm Fv}(\phi)$) and 
a graph $G$, which 
is defined by  following expression
$$
\langle \phi,G\rangle=
\frac{|\{(v_1,\dots,v_{|{\rm Fv}(\phi)|})\in G^{|{\rm Fv}(\phi)|}: G\models\phi(v_1,\dots,v_{|{\rm Fv}(\phi)|})\}|}
{|G|^{|{\rm Fv}(\phi)|}}.
$$
In other words, $\langle \phi,G\rangle$ is the probability that $\phi$ is satisfied in $G$ by a random assignment of vertices (chosen independently and uniformly in the vertex set of $G$) to the free variables of $G$.

Stone pairing induces a notion of convergence: a sequence of graphs $(G_n)_{n\in\bbbn}$ is {\em ${\rm FO}$-convergent} if, for every first order
formula $\phi$ (in the language of graphs), the values $\langle \phi, G_n\rangle$ converge
as $n\rightarrow\infty$.
In other words, $(G_n)_{n\in\bbbn}$ is ${\rm FO}$-convergent if
the probability that a formula $\phi$ is satisfied by the graph $G_n$ with a random assignment of vertices 
of $G_n$ to the free variables of $\phi$ converges as $n$ grows to infinity.

It is sometimes interesting to consider weaker notions of convergence, by restricting the set of considered formulas
to a fragment $X$ of ${\rm FO}$. In this case, we speak about {\em $X$-convergence} instead of ${\rm FO}$-convergence.
Of special importance are the following fragments:
\begin{table}[h!t]%
\begin{tabular}[C]{|l|p{.4\textwidth}|p{.4\textwidth}|}\hlx{hv}
Fragment&Type of formulas&Type of convergence\\ \hlx{vhhv}
${\rm QF}$&{\em quantifier free} formulas&{\em left convergence} \cite{Lov'asz2006}\\ \hlx{vhv}
${\rm FO}_0$&{\em sentences} (no free variables)&{\em elementary convergence}\\ \hlx{vhv}
${\rm FO}_p$&formulas with free variables in $\{x_1,\dots,x_p\}$&\\ \hlx{vhv}
${\rm FO}^{\rm local}$&{\em local} formulas (depending on a fixed 
distance neighborhood of the free variables)&\\ \hlx{vhv}
${\rm FO}_1^{\rm local}$&{\em local} formulas with single free variable&{\em local convergence} (if bounded degree) \cite{Benjamini2001}\\ \hlx{vh}
\end{tabular}
\caption{Fragments of specific importance.
}
\label{tab:frag}
\end{table}

Note that the above notions clearly extend to relational structures. Precisely,
if we consider relational structures with signature $\sig$, the symbols of the relations and constants
in $\sig$ define the  non-logical symbols of the vocabulary of the first-order language ${\rm FO}(\sig)$
associated to $\sig$-structures. Notice that if $\sig$ is at most countable then ${\rm FO}(\sig)$ is countable.
We have shown in \cite{CMUC,NPOM1arxiv} that every finite
relational structure $\mathbf A$ with (at most countable) signature $\lambda$ defines (injectively) a probability measure $\mu_{\mathbf A}$ on the standard Borel space $S(\mathcal B({\rm FO}(\lambda)))$, which is the Stone space of the Lindenbaum-Tarski algebra of first-order formulas (modulo logical equivalence) in the language of $\lambda$-relational structures.
In this setting, a sequence $(\mathbf A_n)_{n\in\bbbn}$ of $\lambda$-structures is ${\rm FO}$-convergent if and only 
if the sequence $(\mu_{\mathbf A_n})_{n\in\bbbn}$ of measures converge (in the sense of a weak-* convergence), and that
the uniquely determined limit probability measure $\mu$ is such that for every first-order formula $\phi$ it holds
$$\lim_{n\rightarrow\infty}\langle\phi,\mathbf A_n\rangle=\int_{S(\mathcal B({\rm FO}(\lambda)))} \mathbf 1_{K(\phi)}(T)\,{\rm d}\mu(T),$$
where $\mathbf 1_{K(\phi)}$ stands for the indicator function of the set $K(\phi)$ of the $T\in S(\mathcal B({\rm FO}(\lambda)))$
such that $\phi\in T$. Note that the space of probability measures on the Stone space of a countable Boolean algebra,
 equipped with the weak topology, is compact.

It is natural to search for a limit object that would more look like a relational structure. Thus we introduced in
\cite{NPOM1arxiv} --- as candidate for a possible limit object of sparse structures --- the notion of modeling,
which extends the notion of graphing introduced for bounded degree graphs.
Here is an outline of its definition.
A {\em relational sample space} is a relational structure $\mathbf A$ (with signature $\lambda$) 
with additional structure:
The domain $A$ of $\mathbf A$ of a relation sample space is a standard Borel space
(with Borel $\sigma$-algebra $\Sigma_{\mathbf A}$)
 with the property that every subset of $A^p$ that is first-order definable 
 in ${\rm FO}(\lambda)$
  is measurable (in $A^p$ with respect to the product $\sigma$-algebra). 
	A {\em modeling} is a relational sample space equiped with a probability measure (denoted $\nu_{\mathbf A}$).
	For brevity we shall use the same letter $\mathbf A$  for  structure, relational sample space, and modeling.
The definition of modelings allows us to extend Stone pairing naturally to modelings:
the {\em Stone pairing} $\langle \phi,\mathbf A\rangle$ of a first-order formula $\phi$ 
(with free variables in $\{x_1,\dots,x_p\}$) and 
a modeling $\mathbf A$, 
is defined by 
$$
\langle \phi,\mathbf A\rangle=\idotsint \mathbf 1_{\Omega_\phi(\mathbf A)}(x_1,\dots,x_p)\,{\rm d}\nu_{\mathbf A}(x_1)\dots{\rm d}\nu_{\mathbf A}(x_p),$$
where $\mathbf 1_{\Omega_\phi(\mathbf A)}$ is the indicator function of the {\em solution set} $\Omega_\phi(\mathbf A)$
of $\phi$ in $\mathbf A$, that is: 
$$\Omega_\phi(\mathbf A)=\{(v_1,\dots,v_p)\in A^p:\quad \mathbf A\models\phi(v_1,\dots,v_p).$$
Note that every finite structure canonically defines a modeling (with same universe, discrete $\sigma$-algebra, and
uniform probability measure) and that in the definition above matches the definition of Stone pairing of a formula
and a finite structure introduced earlier.

In the following, we assume that free variables of formulas are of the form $x_i$ with $i\in\bbbn$.
Note that the free variables need not to be indexed by consecutive integers.
For a formula $\phi$, denote by $\phi^\triangledown$ the formula
obtained by packing  the free variables of $\phi$: if the free variables of $\phi$
are $x_{i_1},\dots,x_{i_p}$ with $i_1<i_2<\dots<i_p$ then $\phi^\triangledown$ is obtained 
from $\phi$ by renaming $x_{i_1},\dots,x_{i_p}$ to $x_{1},\dots,x_{p}$.
 Although
$\Omega_\phi(\mathbf A)$ and $\Omega_{\phi^\triangledown}(\mathbf A)$ differ in general, it is clear
that they have same measure (as $\Omega_\phi(\mathbf A)$ can be obtained from $\Omega_{\phi^\triangledown}(\mathbf A)$
by taking the Cartesian product by some power of $A$, and then permuting the coordinates). Hence for every 
formula $\phi$ it holds
$$\langle\phi,\,\cdot\,\rangle=\langle\phi^{\triangledown},\,\cdot\,\rangle,$$
that is: the Stone pairing is invariant by renaming of the free variables.

The expressive power of the Stone pairing goes slightly beyond satisfaction statistics of first-order formulas.
In particular, we prove (see Corollary~\ref{cor:omega1}) that the Stone pairing $\langle\,\cdot\,,\mathbf A\rangle$ 
can be extended in a unique way to the infinitary language $\mathcal L_{\omega_1\omega}$, which is an extension of
${\rm FO}$ allowing countable conjunctions and disjunctions \cite{scott1958sentential,Tarski1958}. Note that this language is still complete,
as proved by Karp \cite{karp1964languages}. Although the compactness theorem does not hold for $\mathcal L_{\omega_1,\omega}$,
the interpolation theorem for $\mathcal L_{\omega_1,\omega}$ was proved by Lopez-Escobar \cite{Lopez-Escobar1965} and Scott's isomorphism theorem for $\mathcal L_{\omega_1,\omega}$ by Scott \cite{scott1965logic}. 
For a modeling $\mathbf A$ and an integer $p$, the $\mathcal L_{\omega_1,\omega}$-definable subsets of $A^p$ correspond to the smallest $\sigma$-algebra that contains all the first-order definable subsets of $A^p$ (see Lemma~\ref{lem:omega}).
According to the definition of a modeling, this means that all $\mathcal L_{\omega_1,\omega}$-definable sets of a modeling
are Borel measurable.

We say that a class $\mathcal C$ of structures {\em admits modeling limits} if for every ${\rm FO}$-convergent
sequence of structures $\mathbf A_n\in\mathcal C$ there is a modeling $\mathbf L$ such that
for every $\phi\in{\rm FO}$ it holds
$$\langle\phi,\mathbf L\rangle=\lim_{n\rightarrow\infty}\langle\phi,\mathbf A_n\rangle,$$
what we denote by $\mathbf A_n\xrightarrow{\rm FO}\mathbf L$.
More generally, for a fragment $X$ of ${\rm FO}$, we say that a class $\mathcal C$ of structures 
{\em admits modeling $X$-limits} if for every ${X}$-convergent
sequence of structures $\mathbf A_n\in\mathcal C$ there is a modeling $\mathbf L$ such that
for every $\phi\in X$ it holds $\langle\phi,\mathbf L\rangle=\lim_{n\rightarrow\infty}\langle\phi,\mathbf A_n\rangle$,
and we denote this by $\mathbf A_n\xrightarrow{X}\mathbf L$.

The following results have been proved in~\cite{NPOM1arxiv}:
\begin{itemize}
	\item every class of graphs with bounded degree admits modeling limits;
	\item every class of graphs of colored rooted trees bounded height admits modeling limits;
	\item every class of graphs with bounded tree-depth admits modeling limits.
\end{itemize}

On the other hand, only sparse monotone classes of graphs can admits modeling limits.
Precisely, if a monotone class of graphs admits modeling limits, then it is nowhere dense \cite{NPOM1arxiv}, 
and we conjectured that a monotone class of graphs actually admits modeling limits
if and only if it is nowhere dense.

Recall that a monotone class of graphs ${\mathcal C}$ is {\em nowhere dense} if, for every integer $p$
there exists a graph whose $p$-subdivision is not in $\mathcal C$ (for more on nowhere dense graphs, see
\cite{ECM2009,ND_logic,Nevsetvril2010a,ND_characterization,Sparsity}). The importance of nowhere dense classes
and the strong relationship of this notion with first-order logic is examplified by the recent result of
Grohe, Kreutzer, and Siebertz \cite{Grohe2013}, which states that (under a reasonable complexity theoretic assumption)
deciding first-order properties of graphs in a monotone class $\mathcal C$ is fixed-parameter tractable if and only if
$\mathcal C$ is nowhere dense.

In this paper, we initiate a systematic study of hereditary classes that admit modeling limits.
We prove that the problem of the existence of a modeling limit can be reduced to 
the study of ${\rm FO}^{\rm local}$-convergence, and then to two ``typical'' particular cases:
\begin{itemize}
	\item {\em Residual sequences}, that is sequences such that (intuitively) the limit has only zero-measure connected components,
	\item {\em Non-dispersive sequences}, that is sequences such that (intuitively) the limit is (almost) connected.
\end{itemize}

A modeling $\mathbf A$ with universe $A$ satisfies the {\em Finitary Mass Transport Principle} if, 
for every $\phi,\psi\in{\rm FO}_1(\sig)$ and every integers $a,b$ such that
$$\begin{cases}
\phi\entails (\exists^{\geq a}y)\,(x_1\sim y)\wedge\psi(y)\\
\psi\entails (\exists^{\leq b}y)\,(x_1\sim y)\wedge\phi(y)
\end{cases}$$
it holds
$$a\,\langle\phi,\mathbf A\rangle\leq b\,\langle\psi,\mathbf A\rangle.$$
It is clear that every finite structure satisfies the Finitary Mass Transport Principle, hence
every modeling ${\rm FO}$-limit of finite structures satisfies the Finitary Mass Transport Principle, too.

A stronger version of this principle, which is also satisfied by every finite structure, does not
automatically hold in the limit.
A modeling $\mathbf A$ with universe $A$ satisfies the {\em Strong Finitary Mass Transport Principle} if, 
for every measurable subsets $X,Y$ of $A$, and every integers $a,b$, the following property holds:
\begin{quote}
If every $x\in X$ has at least $a$ neighbors in $Y$ and every $y\in Y$ has at most $b$ neighbors in $X$ then
$a\,\nu_{\mathbf A}(X)\leq b\,\nu_{\mathbf A}(Y)$.
\end{quote}

In this context, we prove the following theorem, which is the principal result of this paper.

\begin{theorem}
Let $\mathcal C$ be a hereditary class of structures.

Assume that for every $\mathbf A_n\in\mathcal C$ and every $\rho_n\in A_n$  ($n\in\bbbn$) the following
properties hold:
\begin{enumerate}
	\item  if $(\mathbf A_n)_{n\in\bbbn}$ is  ${\rm FO}_1^{\rm local}$-convergent and residual, then it has
a modeling ${\rm FO}_1^{\rm local}$-limit;
\item  if $(\mathbf A_n,\rho_n)_{n\in\bbbn}$ is ${\rm FO}^{\rm local}$-convergent
(resp. ${\rm FO}_p^{\rm local}$-convergent) and
$\rho$-non-dispersive  then it has
a modeling ${\rm FO}^{\rm local}$-limit (resp. a ${\rm FO}_p^{\rm local}$-limit).
\end{enumerate}

Then $\mathcal C$ admits modeling limits (resp. modeling ${\rm FO}_p$-limits).

Moreover, if in cases (1) and (2) the modeling limits satisfy the Strong Finitary Mass Transport Principle, then
$\mathcal C$ admits modeling limits (resp. modeling ${\rm FO}_p$-limits) that satisfy the Strong Finitary Mass Transport Principle.
\end{theorem}

Then we apply this theorem in Section~\ref{sec:trees} to give a simple proof of the fact that the class of forests admit modeling limits.

\begin{theorem}
The class of finite forests admits modeling limits, that is:
every ${\rm FO}$-convergent sequence of finite forests as a modeling ${\rm FO}$-limit that satisfies the Strong Finitary Mass Transport Principle.
\end{theorem}

Note that a result similar to Theorem~\ref{thm:mtree} was recently claimed by Kr\'a\v l, Kupec, and 
T\r uma \cite{Kravl2013}.

\section{Preliminaries}
Let $\mathbf A$ be a relational structure with signature $\sig$ and universe $A$, and let $X\subseteq A$.
The {\em substructure} $\mathbf A[X]$ {\em induced} by $X$ has domain $X$ and the same relations
as $\mathbf A$ (restricted to $X$). A class $\mathcal C$ of $\sig$-structures is {\em hereditary}
if every induced substructure of a structure in $\mathcal C$ belongs to $\mathcal C$:
$(\forall \mathbf A\in\mathcal C, \forall X\subset A)\ \mathbf A[X]\in\mathcal C$.

The {\em distance} between two vertices $u,v\in A$ is the smallest number of relations 
inducing a connected substructure of $\mathbf A$ and containing both $u$ and $v$, that is the
graph distance between $u$ and $v$ in the Gaifman graph of $\mathbf A$.
For $r\in\bbbn$ and $v\in A$, we denote by $B_r(\mathbf A,v)$ the 
{\em ball of radius $r$ centered at $v$}, that is the substructure of $\mathbf A$
induced by the vertices at distance at most $r$ from $v$ in $\mathbf A$.
More generally, for $v_1,\dots,v_k\in A$, we denote by $B_r(\mathbf A,v_1,\dots,v_k)$ the substructure of $\mathbf A$
induced by the vertices at distance at most $r$ from at least one of the $v_i$ ($1\leq i\leq k$) in $\mathbf A$.

A formula $\phi\in{\rm FO}_p^{\rm local}$ is {\em $r$-local} if its satisfaction only depends on the $r$-neighborhood
of the free variables, that is: for every $\sig$-structure $\mathbf A$ and every $v_1,\dots,v_p\in A^p$ it holds
$$\mathbf A\models \phi(v_1,\dots,v_p)\quad\iff\quad B_r(\mathbf A,v_1,\dots,v_p)\models\phi(v_1,\dots,v_p).$$

Recall the particular case of Gaifman locality theorem for sentences, which we will be usefull in the following.
A {\em local sentence} is a sentence of the form
$$
\exists x_1\dots\exists x_k\ \bigl(\bigwedge_{1\leq i<j\leq k}{\rm dist}(x_i,x_j)>2r\ \wedge\ \bigwedge_{1\leq i\leq k}\psi_i(x_i)\bigr),
$$
where $r,k\geq 1$ and $\psi_i$ is $r$-local.
\begin{theorem}[Gaifman \cite{Gaifman1982}]
\label{thm:gaifman0}
Every first-order sentence is equivalent to a Boolean combination of local sentences.
\end{theorem}

We end this section with two very simple but usefull lemmas.
\begin{lemma}
\label{lem:1}
Let $\phi,\psi$ be formulas. Then it holds
$$|\langle\phi,\,\cdot\,\rangle-\langle\phi\wedge\psi,\,\cdot\,\rangle|\leq 1-\langle\psi,\,\cdot\,\rangle.$$
\end{lemma}
\begin{proof}
$$
\langle \psi\wedge\phi,\,\cdot\,\rangle\leq\langle \phi,\,\cdot\,\rangle=
\langle \neg\psi\wedge\phi,\,\cdot\,\rangle+\langle \psi\wedge\phi,\,\cdot\,\rangle\leq
\langle \neg\psi,\,\cdot\,\rangle+\langle \psi\wedge\phi,\,\cdot\,\rangle,
$$
Thus
$$|\langle\phi,\,\cdot\,\rangle-\langle\phi\wedge\psi,\,\cdot\,\rangle|\leq \langle \neg\psi,\,\cdot\,\rangle=1-\langle\psi,\,\cdot\,\rangle.$$
\end{proof}
\begin{lemma}
\label{lem:indep}
Let $\psi_1,\dots,\psi_p$ be formulas without common free variables.
Then it holds
$$\langle\bigwedge_{i=1}^p\psi_i,\,\cdot\,\rangle=\prod_{i=1}^p\langle\psi_i,\,\cdot\,\rangle.$$
\end{lemma}
\begin{proof}
%For $1\leq i\leq p$, denote by $\psi_i^\triangledown$ the formula
%obtained from $\psi_i$ by renaming  
%$x_1,\dots,x_{{\rm Fv}(\psi_i)}$ the free variables of $\psi_i$.
Let $k=\max\{i: x_i\in\bigcup_{j=1}^p{\rm Fv}(\psi_j)\}-\sum_{j=1}^p|{\rm Fv}(\psi_j)|$.
For every modeling $\mathbf A$, the solution set $\Omega_{\bigwedge_{i=1}^p\psi_i(x_i)}(\mathbf A)$ can be obtained
from $\Omega_{\psi_1^\triangledown}(\mathbf A)\times\dots\times\Omega_{\psi_p^\triangledown}(\mathbf A)\times A^k$ by permuting the coordinates, hence both sets have the same measure, that is:
$$\langle\bigwedge_{i=1}^p\psi_i,\,\cdot\,\rangle=\prod_{i=1}^p\langle\psi_i^\triangledown,\,\cdot\,\rangle=\prod_{i=1}^p\langle\psi_i,\,\cdot\,\rangle.$$
\end{proof}
\section{What does Stone pairing measure?}
By definition, the Stone pairing $\langle \phi,\mathbf A\rangle$
measures the probability that a given first-order formula $\phi$ is satisfied in $\mathbf A$
by a random assignment of vertices of $\mathbf A$ to the free variables. 
For this definition to make sense, we have to assume that 
every subset of a power of $A$ that is first-order definable without parameters is measurable.
Hence we have to consider, for each $p\in\bbbn$, a $\sigma$-algebra
on $A^p$ that contains all subsets of $A^p$ that are first-order definable without parameters.

The aim of this section is to prove that the minimal $\sigma$-algebra including 
all subsets of $A^p$ that are first-order definable without parameters
is exactly the family of all subsets of $A^p$ that are $\mathcal L_{\omega_1\omega}$-definable
without parameters.

We take time out for two lemmas.
\begin{lemma}
\label{lem:sig}
Let $\Omega$ be a set. 
For $p\in\bbbn$, let $\mathcal A_p$ be a field of sets on $\Omega^p$, and let $\sigma(\mathcal A_p)$ be the
minimal $\sigma$-algebra that contains $\mathcal A_p$.
For $p\in\bbbn$, let $f_p:\Omega^{p+1}\rightarrow \Omega^{p}$ and let $F_p:\mathcal P(\Omega^{p+1})\rightarrow\mathcal P(\Omega^{p})$ be defined by $F_p(X)=\{f_p(x): x\in X\}$.

Assume that for each $p\in\bbbn$, $F_p$ maps $\mathcal A_{p+1}$ to $\mathcal A_{p}$.

Then $F_p$ maps $\sigma(\mathcal A_{p+1})$ to $\sigma(\mathcal A_{p})$.
\end{lemma}
\begin{proof}
The proof follows the standard construction of a $\sigma$-algebra by transfinite induction.
For $p\in\bbbn$, we let
\begin{itemize}
	\item $\mathcal S_{p,0}$ be the collection of sets obtained as countable unions of increasing sets in $\mathcal A_p$,
	that is: sets of the form $\bigcup_{i\in\bbbn}X_i$ where $X_i\in\mathcal A_p$ and 
	$X_1\subseteq X_2\subseteq\dots\subseteq X_n\subseteq\dots$;
	\item $\mathcal P_{p,0}$ be the collection of sets obtained as countable intersections of decreasing sets in $\mathcal A_p$,
	that is: sets of the form $\bigcap_{i\in\bbbn}X_i$ where $X_i\in\mathcal A_p$ and 
	$X_1\supseteq X_2\supseteq\dots\supseteq X_n\supseteq\dots$;
	\item (for $i\geq 1$ not a limit ordinal) $\mathcal S_{p,i}$ be the collection of sets obtained as countable unions of increasing sets in $\mathcal P_{p,i-1}$;
	\item (for $i\geq 1$ not a limit ordinal) $\mathcal P_{p,i}$ be the collection of sets obtained as countable intersections of decreasing sets in $\mathcal S_{p,i-1}$;
	\item (for $i$ limit ordinal) $\mathcal S_{p,i}=\bigcup_{j<i}S_{p,j}$ and $\mathcal P_{p,i}=\bigcup_{j<i}P_{p,j}$.
\end{itemize}

Then it is easily checked that by induction that for every $i$ up to $\omega_1$ it holds:
\begin{itemize}
	\item for all $X\in\Omega^p$, $(X\in \mathcal S_{p,i})\iff(\Omega^p\setminus X\in \mathcal P_{p,i})$;
	\item for every limit ordinal $i$, $\mathcal S_{p,i}=\mathcal P_{p,i}$;
	\item if $X,Y\in \mathcal S_{p,i}$ then $X\cup Y\in \mathcal S_{p,i}$ and $X\cap Y \in \mathcal S_{p,i}$;
	\item if $X,Y\in \mathcal P_{p,i}$ then $X\cup Y\in \mathcal P_{p,i}$ and $X\cap Y \in \mathcal P_{p,i}$;
	\item $F_p$ maps $\mathcal S_{p+1,i}$ to $\mathcal S_{p,i}$ and $\mathcal P_{p+1,i}$ to $\mathcal P_{p,i}$;
\end{itemize}

According to the monotone class theorem, $\sigma(\mathcal A_p)=\mathcal S_{p,\omega_1}=\mathcal P_{p,\omega_1}$.
\end{proof}

\begin{lemma}
\label{lem:omega}
We consider a relational structure $\mathbf A$ with countable signature.

Let $\mathcal A_p$ (resp. $\mathcal A_p^+$) be the field of sets of all the subsets of $A^p$ that are first-order definable without (resp. with) parameters.
Then the smallest $\sigma$-algebra $\sigma(\mathcal A_p)\supseteq \mathcal A_p$ (resp. $\sigma(\mathcal A_p^+)\supseteq\mathcal A_p^+$) %that includes $\mathcal A_p$ (resp. $\mathcal A_p^+$)
is the algebra of 
all the subsets of $A^p$ that are definable in $\mathcal L_{\omega_1\omega}$ without (resp. with) parameters.
\end{lemma}
\begin{proof}
Let $\Omega=L$ and
$f_p:A^{p+1}\rightarrow A^p$ be the projection map. According to Lemma~\ref{lem:sig}, the projection
map send sets in $\sigma(\mathcal A_{p+1})$ to sets in $\sigma(\mathcal A_p)$
(and sets in $\sigma(\mathcal A_{p+1}^+)$ to sets in $\sigma(\mathcal A_p^+)$). It follows easily 
that subsets of $A^p$ that are $\mathcal L_{\omega_1\omega}$-definable without (resp. with) parameters are exactly
 those in $\sigma(\mathcal A_p)$ (resp. $\sigma(\mathcal A_p^+)$).
\end{proof}

Note that when $\mathbf A$ is a modeling, the collection of the  subsets of $A^p$ definable in $\mathcal L_{\omega_1\omega}$
without parameters
is the $\sigma$-algebra generated by the projection $\Type{p}{A}:A^p\rightarrow S(\mathcal B({\rm FO}_p))$, mapping
a $p$-tuple of vertices of $A$ to its $p$-type:
a subset $X$ of $A^p$ is definable in $\mathcal L_{\omega_1\omega}$ without parameters 
if and only if it is the preimage by $\Type{p}{A}$ of a Borel subset of  $S(\mathcal B({\rm FO}_p))$
(see \cite{NPOM1arxiv} for detailed definition and analysis of $\Type{p}{A}$).

\begin{corollary}
\label{cor:omega1}
For every modeling $\mathbf A$ the Stone pairing $\langle\,\cdot\,,\mathbf A\rangle$ can be extended in a unique
way to $\mathcal L_{\omega_1\omega}$. 
\end{corollary}

\begin{remark}
Let $\Xi(\sig)$ be the set of all probability measures on the Stone space
$S(\mathcal B({\rm FO}(\sig)))$, and let $\mathcal B$ be the $\sigma$-algebra
generated by evaluation maps $\mu\mapsto\mu(A)$ for $A$ measurable set of
$S(\mathcal B({\rm FO}(\sig)))$. It is well known that $(\Xi(\sig),\mathcal B)$
is a standard Borel space (\cite{Kechris1995}, Sect. $17$.E). 
(Hence the space of all finite $\sig$-structures and their ${\rm FO}$-limits is also
a compact standard Borel space, as it can be identified to a closed subspace of $\Xi(\sig)$.)
The mapping $\mathbf A\mapsto\nu_{\mathbf A}$ embeds the space $M$ of modelings
into $\Xi(\sig)$. The initial topology on $M$ with respect to this mapping is the same as 
the topology induced by Stone pairing. Hence
the mapping $\langle\phi,\,\cdot\,\rangle: M\rightarrow [0,1]$, which maps
$\mathbf A$ to $\langle\phi,\mathbf A\rangle$, is 
continuous for $\phi\in{\rm FO}(\sig)$, and measurable for $\phi\in{\mathcal L_{\omega_1\omega}}(\sig)$.

Also remark that the topology of $\Xi(\sig)$ can be defined by means of
L\'evy--Prokhorov metric (by choosing some metric on the Stone space). 
For instance, for finite signature $\sig$, the topology of $M$ can be generated by the pseudometric:
$$
{\rm dist}(\mathbf A,\mathbf B)=2^{-\sup\{n|\ \forall\phi\in{\rm FO}(\sig),\,
{\rm qrank}(\phi)+|{\rm Fv}(\phi)|\leq n\Rightarrow |\langle\phi,\mathbf A\rangle-\langle\phi,\mathbf B\rangle|<2^{-n}\}}.
$$
\end{remark}

\begin{theorem}
Let $\mathbf A$ be a relational sample space. Then every subset of $A^p$ that is
$\mathcal L_{\omega_1\omega}$-definable  (with parameters) is measurable (with respect to product
Borel $\sigma$-algebra $\Sigma_{\mathbf A}^{\otimes p}$).
\end{theorem}
\subsection{Interpretation Schemes}
Interpretation Schemes (introduced in this setting in~\cite{NPOM1arxiv}) generalize to other logics than ${\rm FO}$.

\begin{definition}
\label{def:interp}
Let $\mathcal L$ be a logic (for us, ${\rm FO}$ or ${\mathcal L}_{\omega_1\omega}$). For $p\in\bbbn$ and 
a signature $\sig$, $\mathcal L_p(\sig)$ denotes the set of the formulas in the language of $\sig$ in logic
$\mathcal L$, with free variables in $\{x_1,\dots,x_p\}$. 

Let $\kappa,\sig$ be signatures, where $\sig$ has 
$q$ relational symbols $R_1,\dots,R_q$ with respective arities $r_1,\dots,r_q$.

An {\em $\mathcal L$-interpretation scheme} ${\mathsf I}$ of $\sig$-structures
in $\kappa$-structures is defined by an integer $k$ --- the {\em exponent} of the $\mathcal L$-interpretation scheme --- a formula
$E\in{\mathcal L}_{2k}(\kappa)$, a formula $\theta_0\in{\mathcal L}_k(\kappa)$, and
 a formula 
$\theta_i\in{\rm FO}_{r_ik}(\kappa)$ for each symbol $R_i\in\sig$, such that:
\begin{itemize}
  \item the formula $E$ defines an equivalence relation
  of $k$-tuples;
   \item each formula $\theta_i$ is compatible with $E$, in the sense that
   for
   every $0\leq i\leq q$ it holds $$
  \bigwedge_{1\leq j\leq r_i}\,E(\mathbf x_j,\mathbf y_j)\quad\entails\quad
  \theta_i(\mathbf x_1,\dots,\mathbf x_{r_i})\leftrightarrow\theta_i(\mathbf
  y_1,\dots,\mathbf y_{r_i}),
   $$
   where $r_0=1$, boldface $\mathbf x_j$ and $\mathbf y_j$ represent
   $k$-tuples of free variables, and 
   where $\theta_i(\mathbf x_1,\dots,\mathbf x_{r_i})$ stands for
 $\theta_i(x_{1,1},\dots,x_{1,k},\dots,x_{r_i,1},\dots,x_{r_i,k})$.
\end{itemize}

For a $\kappa$-structure $\mathbf A$, we denote by $\mathsf{I}(\mathbf A)$ the
$\sig$-structure $\mathbf B$ defined as follows:
 \begin{itemize}
   \item the domain $B$ of $\mathbf B$ is the subset
   of the $E$-equivalence classes $[\mathbf x]\subseteq A^k$  of the tuples $\mathbf x=(x_1,\dots,x_k)$ 
  such that $\mathbf A\models \theta_0(\mathbf x)$;
   \item for each $1\leq i\leq q$ and every 
   $\mathbf v_1,\dots,\mathbf v_{s_i}\in A^{kr_i}$ such that
   $\mathbf A\models\theta_0(\mathbf v_j)$ (for every $1\leq j\leq r_i$) it
   holds
   $$
   \mathbf B\models R_i([\mathbf v_1],\dots,[\mathbf v_{r_i}])\quad\iff\quad
   \mathbf A\models \theta_i(\mathbf v_1,\dots,\mathbf v_{r_i}).
   $$  
 \end{itemize}
\end{definition}

From the standard properties of model theoretical interpretations
(see, for instance \cite{Lascar2009} p.~180), we state the following: if
$\mathsf I$ is an $\mathcal L$-interpretation of $\sig$-structures in $\kappa$-structures,
then there exists a mapping
$\tilde{\mathsf I}:{\mathcal L}(\sig)\rightarrow {\mathcal L}(\kappa)$ (defined by
means of the formulas $E,\theta_0,\dots,\theta_q$ above) such that
for every $\phi\in {\mathcal L}_p(\sig)$, and every $\kappa$-structure $\mathbf A$,
 the following property holds (while letting $\mathbf B=\mathsf I(\mathbf A)$
 and identifying elements of $B$ with the corresponding equivalence classes of $A^k$):

For every $[\mathbf v_1],\dots,[\mathbf v_p]\in B^{p}$ (where $\mathbf v_i=(v_{i,1},\dots,v_{i,k})\in A^k$)
it holds
$$
  \mathbf B\models \phi([\mathbf v_1],\dots,[\mathbf v_p])\quad\iff\quad \mathbf
  A\models \tilde{\mathsf I}(\phi)(\mathbf v_1,\dots,\mathbf v_p).$$
	
It directly follows from the existence of the mapping $\tilde{\mathsf I}$ that
\begin{itemize}
	\item an ${\rm FO}$-interpretation scheme ${\mathsf I}$ of $\sig$-structures in
$\kappa$-structures defines a continuous mapping from $S(\mathcal B({\rm
FO}(\kappa)))$ to $S(\mathcal B({\rm FO}(\sig)))$;
	\item an $\mathcal L_{\omega_1\omega}$-interpretation scheme ${\mathsf I}$ of $\sig$-structures in
$\kappa$-structures defines a measurable mapping from $S(\mathcal B({\rm
FO}(\kappa)))$ to $S(\mathcal B({\rm FO}(\sig)))$.
\end{itemize}

\begin{definition}
Let $\kappa,\sig$ be signatures.
A {\em basic $\mathcal L$-interpretation scheme} ${\mathsf I}$ of $\sig$-structures
in $\kappa$-structures  with {\em exponent} $k$ is defined by a formula 
$\theta_i\in{\mathcal L}_{kr_i}(\kappa)$ for each symbol $R_i\in\sig$ with arity
$r_i$. 

For a $\kappa$-structure $\mathbf A$, we denote by $\mathsf{I}(\mathbf A)$ the
structure with domain $A^k$ such that, for every $R_i\in\sig$ with arity $r_i$ and every
$\mathbf v_1,\dots,\mathbf v_{r_i}\in A^k$ it holds
$$
\mathsf I(\mathbf A)\models R_i(\mathbf v_1,\dots,\mathbf v_{r_i})\quad\iff\quad
\mathbf A\models\theta_i(\mathbf v_1,\dots,\mathbf v_{r_i}).
$$
\end{definition}

It is immediate that every basic $\mathcal L$-interpretation scheme $\mathsf I$ defines
a mapping $\tilde{\mathsf I}:\mathcal L(\sig)\rightarrow \mathcal L(\kappa)$ such
that for every $\kappa$-structure $\mathbf A$, every $\phi\in\mathcal L_p(\sig)$,
and every $\mathbf v_1,\dots,\mathbf v_p\in A^k$ it holds
$$
\mathsf I(\mathbf A)\models \phi(\mathbf v_1,\dots,\mathbf v_{p})\quad\iff\quad
\mathbf A\models\tilde{\mathsf I}(\phi)(\mathbf v_1,\dots,\mathbf v_{p})
$$

We deduce the following general properties:

\begin{lemma}[\cite{NPOM1arxiv}]
\label{lemma:interpFO}
Let $\mathsf I$ be an ${\rm FO}$-interpretation scheme of $\sig$-structures in $\kappa$-structures.

Then, if a sequence $(\mathbf A_n)_{n\in\bbbn}$ of finite $\kappa$-structures
is ${\rm FO}$-convergent then the sequence 
$(\mathsf{I}(\mathbf A_n))_{n\in\bbbn}$ of (finite) $\sig$-structures
is ${\rm FO}$-convergent. 
\end{lemma}

\begin{lemma}
\label{lemma:interpomega}
Let $\mathsf I$ be an  ${\mathcal L}_{\omega_1\omega}$-interpretation scheme 
of $\sig$-structures in $\kappa$-structures.

If $\mathsf I$ is injective and $\mathbf A$ is a relational sample space, then
$\mathsf I(\mathbf A)$ is a relational sample space.

Furthermore, if $\mathsf I$ is a basic ${\mathcal L}_{\omega_1\omega}$-interpretation scheme 
and $\mathbf A$ is a modeling, then
$\mathsf I(\mathbf A)$ is a modeling and for every $\phi\in\mathcal L_p(\sig)$, it holds
$$\langle\phi,\mathsf I(\mathbf A)\rangle=\langle\tilde{\mathsf I}(\phi),\mathbf A\rangle.$$
\end{lemma}
\begin{proof}
Assume $\mathsf I$ is an injective ${\mathcal L}_{\omega_1\omega}$-interpretation scheme and
$\mathbf A$ is a relational sample space.

We first mark all the (finitely many) parameters and reduce to the case where the interpretation has no parameters
(as in the case of ${\rm FO}$-interpretation, see \cite{NPOM1arxiv}.
Let $D$ be the domain of $f$. As $B$ is $\mathcal L$-definable in $\mathbf B$, $D$ is $\mathcal L$-definable in $\mathbf A$ hence
$D\in\Sigma_{\mathbf A}^k$. 
Then $D$ is a Borel sub-space of $A^k$. As $f$ is a bijection from $D$ to $B$, we deduce
that $(B,\Sigma_{\mathbf B})$ is a standard Borel space.
Moreover, as the inverse image of every $\mathcal L$-definable set of $\mathbf B$ is $\mathcal L$-definable in $\mathbf A$, we deduce that $(\mathbf B,\Sigma_{\mathbf B})$ is a $\sig$-relational sample space.

Assume  $\mathsf I$ is a basic ${\mathcal L}_{\omega_1\omega}$-interpretation scheme and
$\mathbf A$ is a modeling. The pushforward of $\nu_{\mathbf A}$ by $\mathsf I$ defines a probability 
measure on $\mathsf I(\mathbf A)$ such that for every $\phi\in\mathcal L_p(\sig)$, it holds
$$
\langle\phi,\mathsf I(\mathbf A)\rangle={\mathsf I}_*(\nu_{\mathbf A})(\Omega_\phi(\mathsf I(\mathbf A)))
=\nu_{\mathbf A}(\Omega_{\tilde{\mathsf I}(\phi)}(\mathbf A))
=\langle\tilde{\mathsf I}(\phi),\mathbf A\rangle.
$$
\end{proof}

\section{Residual Sequences}
Every modeling can be decomposed into countably many connected components with non-zero measure
and an union of connected components with (individual) zero measure.
A {\em residual modeling} is a modeling, all components of which have zero measure. 

\begin{lemma}
\label{lem:res}
A modeling $\mathbf A$ is residual if it holds %and only if all its connected components have zero measure.
$$\forall r\in\bbbn,\quad \sup_{v\in A}\nu_{\mathbf L}(B_r(\mathbf A,v))=0.$$
\end{lemma}
\begin{proof}
Assume %$\mathbf L$ is residual. 
that for every $r\in\bbbn$ it holds $\sup_{v\in L}\nu_{\mathbf L}(B_r(\mathbf A,v))=0$.
For $u\in A$, the connected component $C_u$ of $u$ is
$C_u=\bigcup_{r\in\bbbn} B_r(\mathbf A,u)$. As all these balls are first-order definable (hence measurable)
we deduce
$$\nu_{\mathbf A}(C_u)=\lim_{r\rightarrow\infty}\nu_{\mathbf A}(B_r(\mathbf A,u))=0.$$
It follows that every connected component of $\mathbf L$ has zero-measure, hence $\mathbf A$ is residual.

Conversely, assume  that 
there exists $u\in A$ and $r\in\bbbn$ such that
$\nu_{\mathbf A}(B_r(\mathbf A,u))>0$. Then the connected component of $u$ does not have zero measure, hence
$\mathbf A$ is not residual.
\end{proof}

This equivalence justifies the following notion of residual sequence.

\begin{definition}[Residual sequence]
A sequence $(\mathbf A_n)_{n\in\bbbn}$ of modelings is {\em residual} if
$$\forall r\in\bbbn,\quad \limsup_{n\rightarrow\infty}\sup_{v\in A_n}\nu_{\mathbf A_n}(B_r(\mathbf A_n,v))=0.$$
\end{definition}

\begin{lemma}
\label{lem:2}
Let $\phi\in {\rm FO}_p^{\rm local}$ be $r$-local, and define
the formula 
$$\theta_r(x_1,\dots,x_p):\quad\bigwedge_{1\leq i<j\leq p} {\rm dist}(x_i,x_j)>r.$$

Then there exist $r$-local formulas $\psi_1,\dots,\psi_p\in{\rm FO}_1^{\rm local}$ such that it holds
$$
|\langle\phi,\,\cdot\,\rangle-\prod_{i=1}^p\langle\psi_i,\,\cdot\,\rangle|\leq 2(1-\langle\theta_r,\,\cdot\,\rangle).
$$
\end{lemma}
\begin{proof}
According to Lemma~\ref{lem:1} it holds
$$|\langle\phi,\,\cdot\,\rangle-\langle\phi\wedge\theta_r,\,\cdot\,\rangle|\leq 1-\langle\theta_r,\,\cdot\,\rangle.$$
According to the $r$-locality of $\phi$ there exist $r$-local formulas $\psi_1,\dots,\psi_p\in{\rm FO}_1^{\rm local}$ such that 
$\phi\wedge\theta_r$ is logically equivalent to $\bigwedge_{i=1}^p\psi_i(x_i)\wedge\theta_r$ (where $\psi_i(x_i)$ denotes
the formula $\psi_i$ with free variable $x_1$ renamed $x_i$).
Thus, according to Lemma~\ref{lem:1}, it holds
$$|\langle\bigwedge_{i=1}^p\psi_i(x_i),\,\cdot\,\rangle-\langle\phi\wedge\theta_r,\,\cdot\,\rangle|\leq 1-\langle\theta_r,\,\cdot\,\rangle.$$
As the formulas $\psi_i(x_i)$ use no common free variables, it holds (according to Lemma~\ref{lem:indep}):
$$
\langle\bigwedge_{i=1}^p\psi_i(x_i),\,\cdot\,\rangle=\prod_{i=1}^p\langle\psi_i,\,\cdot\,\rangle.
$$
Hence the result.
\end{proof}
\begin{corollary}
Let $\phi\in {\rm FO}_p^{\rm local}$ be $r$-local.

Then there exist $r$-local formulas $\psi_1,\dots,\psi_p\in{\rm FO}_1^{\rm local}$ such that 
for every modeling $\mathbf A$ it holds
$$
|\langle\phi,\mathbf A\rangle-\prod_{i=1}^p\langle\psi_i,\mathbf A\rangle|< p^2\,\sup_{v\in A}\nu_{\mathbf A}(B_r(\mathbf A,v)).
$$
\end{corollary}
\begin{proof}
Let $\theta_r$ be defined as in Lemma~\ref{lem:2}. By union bound, we get
\begin{align*}
\langle\neg\theta_r,\mathbf A\rangle&=\langle\bigvee_{1\leq i<j\leq p}{\rm dist}(x_i,x_j)\leq r,\mathbf A\rangle\\
&\leq\binom{p}{2}\langle{\rm dist}(x_1,x_2)\leq r,\mathbf A\rangle\\
&=\binom{p}{2}\,\sup_{v\in A}\nu_{\mathbf A}(B_d(\mathbf A,v)),
\end{align*}
and the result follows from Lemma~\ref{lem:2}.
\end{proof}

\begin{lemma}
\label{lem:resloc}
For a residual sequence, % $(\mathbf A_n)_{n\in\bbbn}$, 
${\rm FO}^{\rm local}$-convergence is equivalent to ${\rm
FO_1}^{\rm local}$-convergence.
\end{lemma}
\begin{proof}
Let $(\mathbf A_n)_{n\in\bbbn}$ be a residual sequence.

If $(\mathbf A_n)_{n\in\bbbn}$ is ${\rm FO}^{\rm local}$-convergent, it is ${\rm FO}^{\rm
local}_1$-convergent;

Assume $(\mathbf A_n)_{n\in\bbbn}$ is ${\rm FO}_1^{\rm local}$-convergent,
let $\phi\in{\rm FO}_p^{\rm local}$ be an $r$-local formula, and let 
$\theta_r$ be the formula $\bigwedge_{1\leq i<j\leq p} {\rm dist}(x_i,x_j)>r$.

As $\mathcal C$ is residual, it holds
$$\lim_{n\rightarrow\infty}\langle\theta_r,\mathbf A_n\rangle=1.$$
According to Lemma~\ref{lem:2}, there exist $r$-local formulas 
$\psi_1,\dots,\psi_p\in{\rm FO}_1^{\rm local}$ such that for every $n\in\bbbn$ it holds
$$|\langle\phi,\mathbf A_n\rangle-\prod_{i=1}^p\langle\psi_i,\mathbf A_n\rangle|\leq 2(1-\langle\theta_r,\mathbf A_n\rangle).$$
Hence 
$$\lim_{n\rightarrow\infty}\langle\phi,\mathbf A_n\rangle=\lim_{n\rightarrow\infty}\prod_{i=1}^p\langle\psi_i,\mathbf A_n\rangle
=\prod_{i=1}^p\lim_{n\rightarrow\infty}\langle\psi_i,\mathbf A_n\rangle.$$

Hence for residual sequences, ${\rm FO}^{\rm local}_1$-convergence
implies  ${\rm FO}^{\rm local}$-convergence.
\end{proof}

To every formula $\phi\in{\rm FO}_p^{\rm local}$ and integer $r\in\bbbn$ we associate the formula
$\Lambda_r(\phi)\in{\rm FO}_p^{\rm local}$ defined as
$$(\exists y_1,\dots,y_p)\quad \bigwedge_{i=1}^p({\rm dist}(x_i,y_i)\leq r)\,\wedge\,\phi(y_1,\dots,y_p).$$
\begin{definition}
A modeling $\mathbf A$ is {\em clean} if for every formula $\phi\in{\rm FO}_1^{\rm local}$ it holds
$$
\mathbf A\models (\exists x)\ \phi(x)\qquad\iff\qquad\lim_{r\rightarrow\infty}\langle\Lambda_r(\phi),\mathbf L\rangle>0.
$$
(Note that the right-hand side condition is equivalent to existence of $d$ such that
$\langle\Lambda_r(\phi),\mathbf A\rangle>0$.)
\end{definition}
\begin{lemma}
Let $\mathbf A$ be a residual clean modeling and let $\phi\in{\rm FO}_1^{\rm local}$.

If $\Omega_\phi(\mathbf L)$ is not empty, then it is uncountable.
\end{lemma}
\begin{proof}
Assume $\Omega_\phi(\mathbf A)$ is not empty. As $\mathbf A$ is clean, there exists $r\in\bbbn$ such that 
$\langle\Lambda_r(\phi),\mathbf A\rangle>0$, that is: $\nu_{\mathbf A}(\Omega_{\Lambda_r(\phi)}(\mathbf A))>0$.
Clearly $\Omega_{\Lambda_r(\phi)}(\mathbf A)=\bigcup_{v\in\Omega_\phi(\mathbf A)}B_r(\mathbf A,v)$.
Assume for contradiction that $\Omega_\phi(\mathbf A)$ is countable. Then
$$\nu_{\mathbf A}(\Omega_{\Lambda_r(\phi)}(\mathbf A))=\sum_{v\in\Omega_\phi(\mathbf A)}\nu_{\mathbf A}(B_r(\mathbf A,v)).$$
As $\mathbf A$ is residual, for every $v\in A$ it holds $\nu_{\mathbf A}(B_r(\mathbf A,v))=0$, what contradicts the assumption 
$\nu_{\mathbf A}(\Omega_{\Lambda_r(\phi)}(\mathbf A))>0$.
\end{proof}

Let $X$ be a fragment of ${\rm FO}$.
Two modeling $\mathbf A$ and $\mathbf A'$ are {\em $X$-equivalent} if, for every $\phi\in X$ it holds
$\langle\phi,\mathbf A'\rangle=\langle\phi,\mathbf A\rangle$.
We shall now show how any modeling can be transformed into a residual clean modeling, which is 
${\rm FO}_1^{\rm local}$-equivalent.

\begin{lemma}
\label{lem:mkres}
Let $\mathbf A$ be a modeling. Then there exists a residual modeling $\mathbf A'$ 
that is ${\rm FO}_1^{\rm local}$-equivalent to $\mathbf A$.
\end{lemma}
\begin{proof}
Consider the modeling $\mathbf A'$ with universe $A'=A\times[0,1]$, 
measure $\nu_{\mathbf A'}=\nu_{\mathbf A}\otimes\lambda$ (where $\lambda$ is standard Borel measure of
$[0,1]$) and relations defined as follows:
for every relation $R_i$ of arity $r_i$ it holds
$$
\mathbf A'\models R_i((v_1,\alpha_1),\dots,(v_{r_i},\alpha_{r_i}))\quad\iff\quad
\mathbf A\models R_i(v_1,\dots,v_{r_i})\text{ and }\alpha_1=\dots=\alpha_{r_i}.
$$
Then $\mathbf A'$ is residual and for every $\phi\in{\rm FO}_1^{\rm local}$ it holds 

$$\langle\phi,\mathbf A\rangle=\nu_{\mathbf A'}(\Omega_\phi(\mathbf A'))
=\nu_{\mathbf A'}(\Omega_\phi(\mathbf A)\times [0,1])
=\nu_{\mathbf A}(\Omega_\phi(\mathbf A))
=\langle\phi,\mathbf A'\rangle.$$
\end{proof}

\begin{lemma}
\label{lem:mkclean}
Let $\mathbf A$ be a residual modeling. Then there exists a residual clean modeling $\mathbf A'$ obtained
from $\mathbf A$ by removing a union of connected components with global $\nu_{\mathbf A}$-measure zero.
\end{lemma}
\begin{proof}
Let $\phi\in{\rm FO}_1^{\rm local}$ be such that $\mathbf A\models (\exists x)\phi(x)$
and $\lim_{r\rightarrow\infty}\langle\Lambda_r(\phi),\mathbf A\rangle=0$.
For $v\in A$, denote by $\mathbf A_v$ the connected component of $\mathbf A$ that contains $v$, that is:
$\mathbf A_v=\bigcup_{r\in\bbbn}B_r(\mathbf A,v)$.
Note that if $A\models\phi(v)$ and if $u\in\mathbf A_v$ then $\mathbf A\models\Lambda_r(\phi)(u)$ but
$\langle\Lambda_r(\Lambda_r(\phi)),\mathbf A\rangle=\langle\Lambda_r(\phi),\mathbf A\rangle=0$.

Note that the assumption on $\phi$ rewrites 
as ``$\Omega_\phi(\mathbf A)\neq\emptyset$ while for every 
$v\in\Omega_\phi(\mathbf A)$ it holds $\nu_{\mathbf A}(\mathbf A_v)=0$''.

Then 
\begin{align*}
\nu_{\mathbf A}\bigl(\bigcup_{v\in\Omega_\phi(\mathbf A)}\mathbf A_v\bigr)&=
\nu_{\mathbf A}\bigl(\bigcup_{v\in\Omega_\phi(\mathbf A)}\bigcup_{r\in\bbbn}B_r(\mathbf A,v)\bigr)
=\nu_{\mathbf A}\bigl(\bigcup_{r\in\bbbn}\bigcup_{v\in\Omega_\phi(\mathbf A)}B_r(\mathbf A,v)\bigr)\\
&=\nu_{\mathbf A}\bigl(\bigcup_{r\in\bbbn}\Omega_{\Lambda_r(\phi)}(\mathbf A)\bigr)
=\sum_{r\in\bbbn}\langle\Lambda_r(\phi),\mathbf A\rangle
=0.
\end{align*}

Denote by $\mathcal F$ the set of all $\phi\in{\rm FO}_1^{\rm local}$ such that
$$\mathbf A\models(\exists x)\phi(x)\quad\text{ and }\quad\lim_{r\rightarrow\infty}\langle\Lambda_r(\phi),\mathbf A\rangle=0,$$
and let $\mathbf A'$ be obtained by removing $\bigcup_{\phi\in\mathcal F}\bigcup_{v\in\Omega_\phi(\mathbf A)}\mathbf A_v$ from $\mathbf A$. As for every $\phi\in\mathcal F$ it holds $\nu_{\mathbf A}\bigl(\bigcup_{v\in\Omega_\phi(\mathbf A)}\mathbf A_v\bigr)=0$ and as $\mathcal F$ is countable, the modeling $\mathbf A'$ differs from $\mathbf A$ by a set of connected components
of global measure $0$. Moreover, it is clear that $\mathbf A'$ is clean.
\end{proof}

Recall that ${\rm FO}$-convergence can be decomposed into elementary convergence and ${\rm FO}^{\rm local}$-convergence:
\begin{lemma}[\cite{CMUC,NPOM1arxiv}]
A sequence $(\mathbf A_n)_{n\in\bbbn}$ is ${\rm FO}$-convergent if and only if it is both elementary convergent and
${\rm FO}^{\rm local}$-convergent.

Consequently, a modeling ${\mathbf L}$ is a modeling ${\rm FO}$-limit of a sequence $(\mathbf A_n)_{n\in\bbbn}$
if and only if it is both an elementary limit and a modeling ${\rm FO}^{\rm local}$-limit of it.
\end{lemma}

The interest of residual clean modelings stands in the following.
\begin{lemma}
\label{lem:mkfull}
Let $(G_n)_{n\in\bbbn}$ be a residual ${\rm FO}$-convergent sequence.

If $\mathbf L$ is a residual clean modeling ${\rm FO}_1^{\rm local}$-limit of $(G_n)_{n\in\bbbn}$ 
and $M$ is a countable elementary limit of $(G_n)_{n\in\bbbn}$ then 
 $\mathbf L\cup M$ is a modeling ${\rm FO}$-limit of $(G_n)_{n\in\bbbn}$.
\end{lemma}
\begin{proof}
According to Theorem~\ref{thm:gaifman0}, it is sufficient to check that if $\psi_1,\dots,\psi_n$ are
$r$-local formulas with a single free variable and if we let
$$\phi(x_1,\dots,x_n):\quad\bigwedge_{1\leq i<j\leq n}{\rm dist}(x_i,x_j)>2r\ \wedge\ \bigwedge_{i=1}^n\psi_i(x_i)$$
then it holds
$$\mathbf L\cup M\models (\exists x_1,\dots x_n)\phi(x_1,\dots,x_n)\quad\iff\quad
M\models (\exists x_1,\dots x_n)\phi(x_1,\dots,x_n).$$
But (according to the locality assumptions) it is equivalent to check 
 that for every $r$-local $\psi_1,\dots,\psi_n$ and associated $\phi$ it holds
$$\mathbf L\models (\exists x_1,\dots x_n)\phi(x_1,\dots,x_n)\quad\Rightarrow\quad
M\models (\exists x_1,\dots x_n)\phi(x_1,\dots,x_n).$$
But  if
$\mathbf L\models (\exists x_1,\dots x_n)\phi(x_1,\dots,x_n)$, then
forall  $1\leq i\leq n$ it holds $\mathbf L\models (\exists x)\psi_i(x)$ hence, as $\mathbf L$ is clean,
it that there exists $r_0\in\bbbn$ such that $\langle\Lambda_{r_0}(\psi_i),\mathbf L\rangle>0$
for every $1\leq i\leq n$.
As $\mathbf L$ is residual, this implies $\langle\Lambda_{r_0}(\phi),\mathbf L\rangle>0$.
Thus there exits $n_0$ such that for every $n\geq n_0$ it holds $\langle\Lambda_{r_0}(\phi),G_n\rangle>0$.
In particular the elementary limit $M$ of $G_n$ satisfies 
$(\exists x_1,\dots,x_n)\phi(x_1,\dots,x_n)$.
\end{proof}

\begin{corollary}
\label{cor:res}
A residual ${\rm FO}$-convergent sequence $(G_n)_{n\in\bbbn}$ admits a modeling ${\rm FO}$-limit
if and only if it admits a modeling ${\rm FO}_1^{\rm local}$-limit.
\end{corollary}

In this context, the following conjecture is interesting.
\begin{conjecture}
\label{conj:1}
Every ${\rm FO}$-convergent residual sequence admits
a modeling ${\rm FO}$-limit.
\end{conjecture}
Note that, according to Lemmas~\ref{lem:mkres},~\ref{lem:mkclean} and~\ref{lem:mkfull}, Conjecture~\ref{conj:1} is equivalent to the 
(seemingly weaker) conjecture that every ${\rm FO}_1^{\rm local}$-convergent residual sequence admits
a modeling ${\rm FO}_1^{\rm local}$-limit.
\section{Non-dispersive Sequences}
The notion of residual sequences derives from the notion of residual modelings, that is: modelings without
connected components with non-zero measure). Similarly the notion of non-dispersive sequences derives from the notion
of connected modelings (modulo a zero-measure set).

\begin{definition}
A sequence $(\mathbf A_n)_{n\in\bbbn}$ of modelings is {\em non-dispersive} if
$$\forall\epsilon>0\,\exists d\in\bbbn,\quad \liminf_{n\rightarrow\infty}\sup_{v\in A_n} \nu_{\mathbf A_n}(B_d(\mathbf A_n,v))>1-\epsilon.$$

In the case of rooted structures, we usually want a stronger statement that the structures remain concentrated around their roots:
a sequence $(\mathbf A_n,\rho_n)_{n\in\bbbn}$ of rooted modelings is {\em $\rho$-non-dispersive} if
$$\forall\epsilon>0\,\exists d\in\bbbn,\quad \liminf_{n\rightarrow\infty}\nu_{\mathbf A_n}(B_d(\mathbf A_n,\rho_n))>1-\epsilon.$$
\end{definition}
Note that every $\rho$-non-dispersive sequence is obviously non-dispersive.

\begin{remark}
Let $(\mathbf A_n)_{n\in\bbbn}$ be a non-dispersive ${\rm FO}^{\rm local}$-convergent sequence.
If $(\mathbf A_n)_{n\in\bbbn}$ has a modeling ${\rm FO}^{\rm local}$-limit $\mathbf L$, then
$\mathbf L$ has a full measure connected component, which is also a modeling ${\rm FO}^{\rm local}$-limit
of $(\mathbf A_n)_{n\in\bbbn}$.
\end{remark}

\begin{lemma}
\label{lem:equivr}
Let $(\mathbf A_n,\rho_n)_{n\in\bbbn}$ be a $\rho$-non-dispersive ${\rm FO}_1$-convergent sequence, with
modeling ${\rm FO}_1^{\rm local}$-limit $(\mathbf L,\rho)$ and a countable elementary limit $(M,\varrho)$.

Let $M_\bullet$ and $\mathbf L_\bullet$ be the connected component of the root in $M$ and $\mathbf L$, respectively.
Then $M_\bullet$ and $\mathbf L_\bullet$ are elementarily equivalent.
\end{lemma}
\begin{proof}
As $(\mathbf A_n,\rho_n)_{n\in\bbbn}$ is $\rho$-non-dispersive, $\mathbf L_\bullet$ has 
full measure.
According to Theorem~\ref{thm:gaifman0}, it is sufficient to check that if $\psi_1,\dots,\psi_n$ are
$r$-local formulas with a single free variable and if we let
$$\phi(x_1,\dots,x_n):\quad\bigwedge_{1\leq i<j\leq n}{\rm dist}(x_i,x_j)>2r\ \wedge\ \bigwedge_{i=1}^n\psi_i(x_i)$$
then it holds
$$M_\bullet\models\exists x_1\dots\exists x_n\ \phi(x_1,\dots,x_n)\quad\iff\quad
\mathbf L_\bullet\models\exists x_1\dots\exists x_n\ \phi(x_1,\dots,x_n).$$

Assume $\mathbf L_\bullet\models\exists x_1\dots\exists x_n\ \phi(x_1,\dots,x_n)$.
Let $v_1,\dots,v_n\in L_\bullet$ be such that $\mathbf L_\bullet\models\phi(v_1,\dots,v_n)$.
As $\mathbf L_\bullet$ is a full measure connected component of $\mathbf L$, there exists 
$d\geq\max_{1\leq i\leq n}{\rm dist}(v_i,\rho)$ such that $\nu_{\mathbf L}(B_d(\mathbf L,\rho))\geq 1/2$.
Let $\zeta\in{\rm FO}_1^{\rm local}$ be the $2d$-local formula
$$
\exists y_1\dots\exists y_n\ \bigl(\bigwedge_{i=1}^n{\rm dist}(x_1,y_i)\leq 2d\ \wedge\phi(y_1,\dots,y_n)\big).
$$
Then obviously $\Omega_\zeta(\mathbf L)\supseteq B_d(\mathbf L,\rho)$ hence
$\langle\zeta,\mathbf L\rangle\geq 1/2$. As $\mathbf L$ is a modeling ${\rm FO}_1^{\rm local}$-limit
of $(\mathbf A_n,\rho_n)_{n\in\bbbn}$, there exists $n_0$ such that for every $n\geq n_0$ it holds
$\langle\zeta,\mathbf A_n\rangle\geq 1/4$. In particular, for every $n\geq n_0$ it holds 
$\mathbf A_n\models\exists x_1\dots\exists x_n\ \phi(x_1,\dots,x_n)$ hence
$\mathbf M_\bullet\models\exists x_1\dots\exists x_n\ \phi(x_1,\dots,x_n)$.

Conversely, assume that $\mathbf M_\bullet\models\exists x_1\dots\exists x_n\ \phi(x_1,\dots,x_n)$.
Let $v_1,\dots,v_n\in M_\bullet$ be such that $\mathbf M_\bullet\models\phi(v_1,\dots,v_n)$, and let
$d=\max_{1\leq i\leq n}{\rm dist}(v_i,\varrho)$. As $(M_\bullet,\rho_n)$ is an elementary limit of 
$(\mathbf A_n,\rho_n)_{n\in\bbbn}$, there exists $n_0$ such that for every $n\geq n_0$ it holds
$$
\mathbf A_n\models\exists x_1\dots\exists x_n\ \bigl(\bigwedge_{i=1}^n{\rm dist}(x_i,\rho_n)\leq d\ \wedge\ 
\psi(x_1,\dots,x_n)\bigr).
$$
As $(\mathbf A_n,\rho_n)_{n\in\bbbn}$ is $\rho$-non-dispersive, there exists $D\geq d$ and $n_1\geq n_0$ such that
for every $n\geq n_0$, it holds $|B_D(\mathbf A_n,\rho_n)|\geq |\mathbf A_n|/2$.
Let $\zeta\in{\rm FO}_1^{\rm local}$ be the $2D$-local formula
$$
\exists y_1\dots\exists y_n\ \bigl(\bigwedge_{i=1}^n{\rm dist}(x_1,y_i)\leq 2D\ \wedge\phi(y_1,\dots,y_n)\big).
$$
Then obviously $\Omega_\zeta(\mathbf A_n)\supseteq B_D(\mathbf A_n,\rho_n)$ hence
$\langle\zeta,\mathbf A_n\rangle\geq 1/2$. As $\mathbf L$ is a modeling ${\rm FO}_1^{\rm local}$-limit
of $(\mathbf A_n,\rho_n)_{n\in\bbbn}$, it holds $\langle\zeta,\mathbf L\rangle\geq 1/2$ hence,
as $\mathbf L_\bullet$ is full dimensional, all of $\Omega_\zeta(\mathbf L)$ but a subset 
with $\nu_{\mathbf L}$-measure zero is included in $L_\bullet^n$. Hence
$\mathbf L_\bullet\models\exists x_1\dots\exists x_n\ \phi(x_1,\dots,x_n)$.
\end{proof}
\begin{lemma}
\label{lem:mkfullc}
Let $p\geq 1$ and let $(\mathbf A_n,\rho_n)_{n\in\bbbn}$ be a $\rho$-non-dispersive ${\rm FO}_p$-convergent sequence, with
modeling ${\rm FO}_p^{\rm local}$-limit $(\mathbf L,\rho)$ and countable elementary limit $(M,\varrho)$.

Let $M_\bullet$ and $\mathbf L_\bullet$ be the connected components of the root in $M$ and $\mathbf L$, respectively.
Then $\mathbf L_\bullet\cup(M\setminus M_\bullet)$ is a modeling ${\rm FO}_p$-limit of $(\mathbf A_n,\rho_n)_{n\in\bbbn}$.
\end{lemma}
\begin{proof}
According to Lemma~\ref{lem:equivr}, $\mathbf L_\bullet$ and $M_\bullet$ are elementarily equivalent, so
$\mathbf L_\bullet\cup(M\setminus M_\bullet)$ and $M$ are elementarily equivalent.
 It follows that $\mathbf L_\bullet\cup(M\setminus M_\bullet)$ is both
an elementary limit of $(\mathbf A_n,\rho_n)_{n\in\bbbn}$ and
a modeling ${\rm FO}_p^{\rm local}$-limit of $(\mathbf A_n,\rho_n)_{n\in\bbbn}$ (as it differs from $\mathbf L$ by 
a set of connected components with global measure zero) hence a 
modeling ${\rm FO}_p$-limit of $(\mathbf A_n,\rho_n)_{n\in\bbbn}$.
\end{proof}

\begin{problem}
Let $(\mathbf A_n)_{n\in\bbbn}$ be a non-dispersive ${\rm FO}_1^{\rm local}$-convergent sequence with modeling
${\rm FO}_1^{\rm local}$-limit $\mathbf L$.
Does there exist $\rho_n\in A_n$ and $\rho\in L$ such that $(\mathbf A_n,\rho_n)_{n\in\bbbn}$ is a 
$\rho$-non-dispersive ${\rm FO}_1^{\rm local}$-convergent sequence with modeling ${\rm FO}_1^{\rm local}$-limit $(\mathbf L,\rho)$?
\end{problem}

Sometimes, $\rho$-non-dispersive sequences may be still quite difficult to handle, and sequences with bounded
diameter may be more tractable. It is thus natural to consider to what extent it could be possible to further
reduce to the bounded diameter case. We give here a partial answer.

\begin{lemma}
\label{lem:boundedh}
Let $(\mathbf A_n,\rho_n)$ be a $\rho$-non-dispersive ${\rm FO}_p^{\rm local}$-convergent sequence, and
let $(\mathbf L,\rho)$ be a connected modeling.

Assume that for each $d\in\bbbn$, it holds that $(B_d(\mathbf L,\rho),\rho)$ is a modeling ${\rm FO}_p^{\rm local}$-limit
of $(B_d(\mathbf A_n,\rho_n),\rho_n)$. Then $(\mathbf L,\rho)$ is a modeling ${\rm FO}_p^{\rm local}$-limit
of $(\mathbf A_n,\rho_n)$.
\end{lemma}
\begin{proof}
Let $\phi\in{\rm FO}_p^{\rm local}$ be $r$-local and let $\epsilon>0$. As $(\mathbf A_n,\rho_n)_{n\in\bbbn}$ is
$\rho$-non-dispersive and as $\mathbf L$ is connected, 
there exist $d,n_0\in\bbbn$ such that $\nu_{\mathbf L}(B_d(\mathbf L,\rho))>1-\epsilon$ and such that
for every $n\geq n_0$ it holds
$|B_d(\mathbf A_n,\rho_n)|>(1-\epsilon) |A_n|$.
Let $\theta$ be the formula ${\rm dist}(x_1,\rho)\leq d-r$, and let $\theta^{(p)}$ be the formula
$\bigwedge_{i=1}^p\theta_d(x_i)$. 
Note that 
$$\langle \theta,B_d(\mathbf A_n,\rho_n)\rangle=\frac{|B_{d-r}(\mathbf A_n,\rho_n)|}{|B_{d}(\mathbf A_n,\rho_n)|}\geq \frac{|B_{d-r}(\mathbf A_n,\rho_n)|}{|A_n|}=\langle \theta,\mathbf A_n\rangle.$$
According to Lemma~\ref{lem:1} it holds
$$|\langle\phi,\mathbf A_n\rangle-\langle\phi\wedge\theta^{(p)},\mathbf A_n \rangle|\leq 1-\langle\theta,\mathbf A_n\rangle^p< p\epsilon.$$
and also
$$|\langle\phi,B_d(\mathbf A_n,\rho_n)\rangle-\langle\phi\wedge\theta^{(p)},B_d(\mathbf A_n,\rho_n) \rangle|\leq 1-\langle\theta,B_d(\mathbf A_n,\rho_n)\rangle^p< p\epsilon.$$
According to the $r$-locality of $\phi$, it holds $\langle\phi\wedge\theta^{(p)},\mathbf A_n \rangle=\langle\phi\wedge\theta^{(p)},B_d(\mathbf A_n,\rho_n) \rangle$, hence
$$|\langle\phi,\mathbf A_n\rangle-\langle\phi,B_d(\mathbf A_n,\rho_n)\rangle|< 2p\epsilon.
$$
Similarly, it holds
$$|\langle\phi,\mathbf L\rangle-\langle\phi,B_d(\mathbf L,\rho)\rangle|< 2p\epsilon.
$$
By assumption, $B_d(\mathbf L,\rho)$ is a modeling ${\rm FO}_p^{\rm local}$-limit of $(B_d(\mathbf A_n,\rho_n))_{n\in\bbbn}$. Thus there exists $n_1\geq n_0$ such that 
$|\langle\phi,B_d(\mathbf L,\rho)\rangle-\langle\phi,B_d(\mathbf A_n,\rho_n)\rangle|< p\epsilon$, hence
for every $n\geq n_1$ it holds
$$
|\langle\phi,\mathbf L\rangle-\langle\phi,\mathbf A_n\rangle|< 5p\epsilon.
$$
Considering $\epsilon\rightarrow 0$, we deduce
$$\langle\phi,\mathbf L\rangle=\lim_{n\rightarrow\infty}\langle\phi,\mathbf A_n\rangle.$$
\end{proof}
\section{Breaking}
The aim of this section is to prove that the study of ${\rm FO}$-convergent sequences of structures in 
a hereditary class naturally reduces to the study of residual sequences and $\rho$-non-dispersive sequences
in that class.

Advancing the main result of this section, Theorem~\ref{thm:break}, we state four technical lemmas.

\begin{lemma}
\label{lem:b1}
Let $0<\epsilon\leq 1$ and let $r\in\bbbn$.

Then, for every graph $G$ there exists a subset $A$ of vertices such that
\begin{enumerate}
	\item\label{enum:l11} for every $a\in A$, it holds $|B_{2r}(G,a)|\geq\epsilon |G|$;
	\item\label{enum:l12} for every $v\notin\bigcup_{a\in A}B_{2r}(G,a)$, it holds $|B_{r}(G,v)|<\epsilon |G|$;
	\item\label{enum:l13} $|A|\leq 1/\epsilon$.
\end{enumerate}
\end{lemma}
\begin{proof}
Let $A$ be a maximal subset of vertices of $G$ such that 
\begin{itemize}
	\item for every $a\in A$, it holds $|B_{r}(G,a)|\geq\epsilon |G|$ (Hence \eqref{enum:l11} holds);
	\item for every distinct $x,y\in A$, it holds $B_{r}(G,x)\cap B_{r}(G,y)=\emptyset$.
\end{itemize}
Obviously, $|A|\leq 1/\epsilon$, hence \eqref{enum:l13} holds. 
Moreover, by maximality of $X$, if $v$ is any vertex such that $|B_{r}(G,v)|\geq\epsilon |G|$ then 
there exists $a\in A$ such that $B_{r_0}(G,a)\cap B_{r_0}(G,v)\neq\emptyset$, thus
$v\in B_{2r}(G,a)$. Hence \eqref{enum:l12} holds.
\end{proof}

\begin{lemma}
\label{lem:b2}
Let $\epsilon>0$, let $r\in\bbbn$ and let $(G_n)_{n\in\bbbn}$ be an ${\rm FO}^{\rm local}$-convergent sequence.
Then there exists integer $q\leq 1/\epsilon$, integer $D$, and increasing function $N:\bbbn\rightarrow\bbbn$ and an  ${\rm FO}^{\rm local}$-convergent sequence $(G_n^+)_{n\in\bbbn}$ of $q$-rooted graphs, such that $G_n^+$ is a $q$-rooting of $G_{N(n)}$ (with roots $c_1^n,\dots,c_q^n$) and
\begin{itemize}
	\item $\lim_{n\rightarrow\infty} {\rm dist}_{G_n^+}(c_i^n,c_j^n)=\infty$ for every $1\leq i<j\leq q$;
	\item $B_D(G_n^+,c_i^n)\cap B_D(G_n^+,c_j^n)=\emptyset$ for every $1\leq i<j\leq q$ and every $n\in\bbbn$;
	\item $|B_D(G_n^+,c_i^n)|\geq\epsilon |G_n^+|$ for every $1\leq i\leq q$ and every $n\in\bbbn$;
	\item $|B_r(G_n^+,v)|<\epsilon |G_n^+|$ for every $v\notin\bigcup_{i=1}^q B_D(G_n^+,c_i^n)$ and every $n\in\bbbn$.
\end{itemize}
\end{lemma}
\begin{proof}
Consider the signature obtained by adding $K=\lfloor 1/\epsilon\rfloor$ unary symbols $R_1,\dots,R_{K}$.
According to Lemma~\ref{lem:b1}, there exists, for each $G_n$, vertices
$z_1^n,\dots,z_{k_n}^n$ such that
\begin{itemize}
	\item for every $1\leq i\leq k_n$, it holds $|B_{2r}(G_n,z_i^n)|\geq\epsilon |G_n|$;
	\item for every $v\notin\bigcup_{i=1}^{k_n}B_{2r}(G_n,z_i^n)$, it holds $|B_{r}(G_n,v)|<\epsilon |G_n|$;
	\item $k_n \leq K$.
\end{itemize}
We mark vertex $z_1^n,\dots,z_{k_n}^n$ by $R_1,\dots,R_{k_n}$ thus obtaining a structure $\mathbf A_n$.
By compactness, the sequence $(\mathbf A_n)_{n\in\bbbn}$ has an ${\rm FO}$-converging subsequence 
$(\mathbf A_{N_1(n)})_{n\in\bbbn}$. Moreover, as the number of roots of $\mathbf A_n$ converges (by elementary convergence),
we can assume without loss of generality that the subsequence is such that all the structures $\mathbf A_{N_1(n)}$
use exactly the marks $R_1,\dots,R_p$ (with $p\leq K$). According to the elementary convergence of $(\mathbf A_{N_1(n)})_{n\in\bbbn}$, the limit 
$d_{i,j}=\lim_{n\rightarrow\infty}{\rm dist}_{\mathbf A_{N_1(n)}}(z_i^{N_1(n)},z_j^{N_1(n)})$ exists for every $1\leq i<i\leq p$ 
and this limit can be either an integer or $\infty$.
Let $I$ be a maximal subset of $\{1,\dots,p\}$ such that $d_{i,j}=\infty$ for every distinct $i,j\in I$, and let $q=|I|$.
Without loss of generality, we assume that $I=\{1,\dots,q\}$.
Define 
$$D=2r+\max\{d_{i,j}| 1\leq i<j\leq p\text{ and }d_{i,j}<\infty\}.$$
Then each ball $B_{2r}(\mathbf A_{N_1(n)},z_i^{N_1(n)})$ with $q<i\leq p$ is included in one of the balls
$B_{D}(\mathbf A_{N_1(n)},z_j^{N_1(n)})$ with $1\leq j\leq q$.
As $d_{i,j}=\infty$ for every $1\leq i<j\leq q$, there exists $n_0$ such that for every $n\geq n_0$ it holds
${\rm dist}_{\mathbf A_{N_1(n)}}(z_i^{N_1(n)},z_j^{N_1(n)})>2D$.
We let $N(n)=N_1(n+n_0)$, $c_i^n=z_i^{N(n)}$, and $G_n^+=\mathbf A_{N(n)}$.
\end{proof}
Let $0<\epsilon\leq 1$, let $r,D\in\bbbn$, and
let $(G_n^+)_{n\in\bbbn}$ be an ${\rm FO}$-convergent sequence of $q$-rooted graphs (with roots $c_i^n$) such that
\begin{itemize}
	\item $\lim_{n\rightarrow\infty} {\rm dist}_{G_n^+}(c_i^n,c_j^n)=\infty$ for every $1\leq i<j\leq q$;
	\item $B_D(G_n^+,c_i^n)\cap B_D(G_n^+,c_j^n)=\emptyset$ for every $1\leq i<j\leq q$ and every $n\in\bbbn$;
	\item $|B_D(G_n^+,c_i^n)|\geq\epsilon |G_n^+|$ for every $1\leq i\leq q$ and every $n\in\bbbn$;
	\item $|B_r(G_n^+,v)|<\epsilon |G_n^+|$ for every $v\notin\bigcup_{i=1}^q B_D(G_n^+,c_i^n)$ and every $n\in\bbbn$.
\end{itemize}
For $1\leq i\leq q$, we define the function
$f_i:\bbbn\rightarrow\bbbr^+$ by
$$f_i(d)=\lim_{n\rightarrow\infty}\frac{|B_d(G_n^+,c_i^n)|}{|G_n^+|}.$$
and we let $\lambda_i=\lim_{d\rightarrow\infty}f_i(d)$.

For $1\leq i\leq q$, we also define $g_i:\bbbn\times(0,1)\rightarrow\bbbn$ by
$$g_i(d,x)=\min\biggl\{n_0:\quad (\forall n\geq n_0)\ \biggl|\frac{|B_d(G_n^+,c_i^n)|}{|G_n^+|}-f_i(d)\biggr|<x\biggr\}.$$

We further define the function $h:\bbbn\rightarrow\bbbn$ by
$$h(x)=\min\{d:\ f_i(d)\geq\lambda_i-x\},$$
and the function $w:\bbbn\rightarrow\bbbn$ by
$$w(d)=\min\{n_0:\quad(\forall n\geq n_0)\ (\forall 1\leq i<j\leq q)\ {\rm dist}_{G_n^+}(c_i^n,c_j^n)>2d\}.$$

\begin{lemma}
For every $\epsilon'>0$ and $r'\in\bbbn$ there exist $d_0,n_0\in\bbbn$ such that for every $n\geq n_0$ it holds
\begin{itemize}
	\item $f_i(d_0-r')\geq \lambda_i-\epsilon'$,
	\item $|B_{d_0-r'}(G_n^+,c_i^n)-f_i(d_0-r')|<\epsilon' |G_n^+|$,
	\item $|B_{d_0}(G_n^+,c_i^n)-f_i(d_0)|<\epsilon' |G_n^+|$,
	\item $|B_{d_0+r'}(G_n^+,c_i^n)-f_i(d_0+r')|<\epsilon' |G_n^+|$,
	\item $|B_{d_0+r'}(G_n^+,c_i^n)\setminus B_{d_0-r'}(G_n^+,c_i^n)|<\epsilon'|G_n^+|$,
	\item the $c_i^n$'s are pairwise at distance strictly greater than $2d_0+4r'$.
\end{itemize}
\end{lemma}
\begin{proof}
Choose $d_0\geq D+2r$ and $\geq r+\max_i h_i(\epsilon'/3)$,
$n_0\geq\max_i \max(g_i(d_0-r,\epsilon'/3),g_i(d_0,\epsilon'/3),g_i(d_0+r,\epsilon'/3))$ and 
$n_0\geq w(d_0-r)$. Note that the third condition then follows from
\begin{equation*}
\begin{split}
\frac{|B_{d_0+r}(G_n^+,c_i^n)\setminus B_{d_0-r}(G_n^+,c_i^n)|}{|G_n^+|}&\leq
\biggl|\frac{|B_{d_0+r}(G_n^+,c_i^n)}{|G_n^+|}-f_i(d_0+r)\biggr|\\
&\quad+\biggl|\frac{|B_{d_0-r}(G_n^+,c_i^n)}{|G_n^+|}-f_i(d_0-r)\biggr|\\
&\quad+f_i(d_0+r)-f_i(d_0-r)
\end{split}
\end{equation*}
and the obvious inequality $f_i(d_0+r)-f_i(d_0-r)\leq\lambda_i-f_i(d_0-r)$.
\end{proof}
\begin{lemma}
For every $\epsilon'>0$ and $r'\in\bbbn$ there exist $d_0,n_0\in\bbbn$ with the following properties:
Define rooted graphs $H_{i,n}=(B_{d_0}(G_n^+,c_i^n),c_i^n)$
and unrooted graph $R_n=G_n^+\setminus\bigcup_i H_{i,n}$, let $G_n^*=R_n\cup\bigcup_i H_{i,n}$,
and let $G_n'={\rm Unmark}(G_n^*)$.
Then for every $n\geq n_0$:
\begin{itemize}
	\item $|H_{i,n}-\lambda_i|G_n^+||<\epsilon' |G_n^+|$,
	\item $|B_{d_0}(H_{i,n},c_i^n)|>(1-\epsilon'/\epsilon) \lambda_i |G_n^+|$,
	\item for every vertex $v\in R_n$ it holds
	$\lambda_0 |B_r(R_n,v)|<\epsilon |G_n^+|$,
	\item for every $r'$-local $\phi\in{\rm FO}_p^{\rm local}$ it holds
	$|\langle\phi,G_n\rangle-\langle\phi,G_n'\rangle|<p\epsilon'/\epsilon$.
\end{itemize}
\end{lemma}
\begin{proof}
Obviously, $\langle\phi,G_n\rangle=\langle\phi,G_n^+\rangle$ and 
$\langle\phi,G_n'\rangle=\langle\phi,G_n^*\rangle$.

Let $\theta$ be the formula
defined as 
$$\bigl(\bigvee_i {\rm dist}(x_1,c_i)\leq d_0-r'\bigr)\vee
\bigl(\bigwedge_i {\rm dist}(x_1,c_i)> d_0+r'\bigr),$$
and let $\theta^{(p)}$ be the formula $\bigwedge_{i=1}^p\theta(x_i)$.

For every $n\geq n_0$, it holds 
$\langle\theta^{(p)},G_n^+\rangle=\langle\theta,G_n^+\rangle^p>(1-\epsilon'/\epsilon)^p$ hence
$1-\langle\theta^{(p)},G_n^+\rangle<p\epsilon'/\epsilon$.
Thus 
$$|\langle\phi,G_n^+\rangle-\langle\phi\wedge\theta^{(p)},G_n^+\rangle|\leq 1-\langle\theta^{(p)},G_n^+\rangle<p\epsilon'/\epsilon.$$

Also, $\langle\theta^{(p)},G_n^*\rangle=\langle\theta,G_n^*\rangle^p>(1-\epsilon'/\epsilon)^p$ hence
$1-\langle\theta^{(p)},G_n^*\rangle<p\epsilon'/\epsilon$.
Thus
$$|\langle\phi,G_n^*\rangle-\langle\phi\wedge\theta^{(p)},G_n^*\rangle|\leq 1-\langle\theta^{(p)},G_n^*\rangle<p\epsilon'/\epsilon$$

According to the $r'$-locality of $\phi$ it holds
$\langle\phi\wedge\theta^{(p)},G_n^+\rangle=\langle\phi\wedge\theta^{(p)},G_n^*\rangle$.
Hence $|\langle\phi,G_n\rangle-\langle\phi,G_n'\rangle|<2p\epsilon'/\epsilon$.
\end{proof}

Now for $a\in\bbbn$ we let $\epsilon'=\epsilon/a, r'=a$, $\hat H_{i,a}=H_{i,n_0(a)}$,
$\hat R_{a}=R_{n_0(a)}$ and $\hat G_a=G_{n_0(a)}'$.
Then it holds for every $n\in\bbbn$:
\begin{itemize}
	\item $(\hat H_{i,n})_{n\in\bbbn}$ is $\rho$-non-dispersive,
	\item $\lim_{n\rightarrow\infty} |\hat H_{i,n}|/|\hat G_n|=\lambda_i$,
	\item for every vertex $v\in R_n$ it holds
	$\lambda_0 |B_r(\hat R_n,v)|<\epsilon |G_n|$,
	\item for every $r$-local $\phi\in{\rm FO}_p^{\rm local}$ and every $n\geq r$ it holds
	$|\langle\phi,G_n\rangle-\langle\phi,\hat G_n\rangle|<p/n$ thus
	$(G_{n})_{n\in\bbbn}$ and $(\hat G_{n})_{n\in\bbbn}$ have the same ${\rm FO}^{\rm local}$-limit.
\end{itemize}

\begin{theorem}
\label{thm:break}
Let $(G_n)_{n\in\bbbn}$ be an ${\rm FO}_p^{\rm local}$-convergent sequence. Then there exist
a $\rho$-non-dispersive ${\rm FO}_p^{\rm local}$-convergent sequences $(\hat H_{i,n})_{n\in\bbbn}$ of rooted graphs
($i\in\bbbn$), a residual ${\rm FO}_p^{\rm local}$-convergent sequence $(\hat R_n)_{n\in\bbbn}$, 
positive real numbers $\lambda_i>0$, and an increasing function $\varphi:\bbbn\rightarrow\bbbn$ such that:
\begin{enumerate}
	\item  The sequences $(G_{n})_{n\in\bbbn}$ and $(\hat G_{n})_{n\in\bbbn}$ have the same ${\rm FO}^{\rm local}$-limit, where
$\hat G_n=R_n\cup\bigcup_{i\in\bbbn}{\rm Unmark}(\hat H_{i,n})$. 
\item If $(\hat H_{i,n})_{n\in\bbbn}$ has  ${\rm FO}_p^{\rm local}$-modeling limit $\mathbf L_i$,
$(\hat R_n)_{n\in\bbbn}$ has ${\rm FO}_1^{\rm local}$-modeling limit $\mathbf L_0$,  $\lambda_0=1-\sum_{i\in\bbbn}\lambda_i$, and $\mathbf L=\coprod_{i\geq 0}(\mathbf L_i,\lambda_i)$, then ${\rm Unmark}(\mathbf L)$  is an ${\rm FO}_p^{\rm local}$-modeling limit of $(G_n)_{n\in\bbbn}$.
\item Furthermore, if $(G_n)_{n\in\bbbn}$ is ${\rm FO}_p$-convergent, then it has a modeling ${\rm FO}_p$-limit, which can be obtained by first cleaning $\mathbf L_0$, computing $\mathbf L$, taking the disjoint union with some countable graph, and then forgetting marks.
\end{enumerate}
\end{theorem}

Our main result immediately follows from Theorem~\ref{thm:break}

\begin{maintheorem}{1}
\label{thm:main}
Let $\mathcal C$ be a hereditary class of structures.

Assume that for every $\mathbf A_n\in\mathcal C$ and every $\rho_n\in A_n$  ($n\in\bbbn$) the following
properties hold:
\begin{enumerate}
	\item  if $(\mathbf A_n)_{n\in\bbbn}$ is  ${\rm FO}_1^{\rm local}$-convergent and residual, then it has
a modeling ${\rm FO}_1^{\rm local}$-limit;
\item  if $(\mathbf A_n,\rho_n)_{n\in\bbbn}$ is ${\rm FO}^{\rm local}$-convergent
(resp. ${\rm FO}_p^{\rm local}$-convergent) and
$\rho$-non-dispersive  then it has
a modeling ${\rm FO}^{\rm local}$-limit (resp. a ${\rm FO}_p^{\rm local}$-limit).
\end{enumerate}

Then $\mathcal C$ admits modeling limits (resp. modeling ${\rm FO}_p$-limits).

Moreover, if in cases (1) and (2) the modeling limits satisfy the Strong Finitary Mass Transport Principle, then
$\mathcal C$ admits modeling limits (resp. modeling ${\rm FO}_p$-limits) that satisfy the Strong Finitary Mass Transport Principle.
\end{maintheorem}
\section{Extended Comb Lemma}

\begin{definition}
A {\em component relation system} for a class $\mathcal C$ of modelings is a sequence
$\varpi_d$ of equivalence relations such that for every $d\in\bbbn$ and for 
every $\mathbf A\in\mathcal C$ there is a partition
of the $\varpi_d$-equivalence classes of $A$ into two parts $\mathcal E_0(\varpi_d,\mathbf A)$ and
$\mathcal E_+(\varpi_d,\mathbf A)$ such that:
\begin{itemize}
	\item every class in $\mathcal E_0(\varpi_d,\mathbf A)$ is a singleton;
	\item $\nu_{\mathbf A}(\bigcup\mathcal E_0(\varpi_d,\mathbf A))<\epsilon(d)+\eta(|A|)$ (where $\lim_{d\rightarrow\infty}\epsilon(d)=\lim_{n\rightarrow\infty}\eta(n)=0$);
	\item two vertices $x,y$ in $\bigcup\mathcal E_+(\varpi_d,\mathbf A)$ belong to a same connected component of $\mathbf A$ if and only if $\mathbf A\models\varpi_d(x,y)$ (i.e. iff $x$ and $y$ belong to a same class).
\end{itemize}
\end{definition}
\begin{definition}[\cite{NPOM1arxiv}]
A family of sequence $(\mathbf A_{i,n})_{n\in\bbbn}\ (i\in I)$ 
of $\sig$-structures is
{\em uniformly elementarily convergent} if, for every formula
$\phi\in{\rm FO}_1(\sig)$ there is an integer $N$ such that it holds
$$
\forall i\in I,\ \forall n'\geq n\geq N,\quad 
(\mathbf A_{i,n}\models (\exists x)\phi(x))\Longrightarrow
(\mathbf A_{i,n'}\models (\exists x)\phi(x)).
$$
\end{definition}
(Note that if a family $(\mathbf A_{i,n})_{n\in\bbbn}\ (i\in I)$
of sequences is uniformly elementarily convergent, then
each sequence $(\mathbf A_{i,n})_{n\in\bbbn}$ is elementarily convergent.)

The proof of the next theorem essentially follows
the lines of Section 3.3 of~\cite{NPOM1arxiv}.
We do not provide the updated version of the proof here, as the proof is quite 
long and technical, but does not present particular additional difficulties when
compared to the original version.

\begin{theorem}[Extended comb structure]
\label{thm:ecomb}
Let $(\mathbf A_n)_{n\in\bbbn}$ be an ${\rm FO}^{\rm local}$-convergent sequence of
finite $\sig$-structures with component relation system $\varpi_d$.

Then there exist $I\subseteq\bbbn\cup\{0\}$ and, for each $i\in I$, a real $\alpha_i$ and a sequence
$(\mathbf B_{i,n})_{n\in\bbbn}$ of $\sig$-structures, such that
$\mathbf A_n=\bigcup_{i\in I}\mathbf B_{i,n}$ (for all $n\in\bbbn$), $\sum_{i\in{I}}\alpha_i=1$, and 
for each $i\in I$ it holds
\begin{itemize}
	\item $\alpha_i=\lim_{n\rightarrow\infty}\frac{|\mathbf B_{i,n}|}{|\mathbf A_n|}$, and
	$\alpha_i>0$ if $i\neq 0$;
	\item if $i=0$ and $\alpha_0>0$ then $(\mathbf B_{i,n})_{n\in\bbbn}$ is ${\rm FO}^{\rm local}$-convergent and
	residual;
\item if $i>0$ then $(\mathbf B_{i,n})_{n\in\bbbn}$ is ${\rm FO}^{\rm local}$-convergent
	 and non-dispersive.
	\end{itemize}
Moreover, if $(\mathbf A_n)_{n\in\bbbn}$ is ${\rm FO}$-convergent we can require
the family $\{(\mathbf B_{i,n})_{n\in\bbbn}: i\in I\}$ to be uniformly elementarily-convergent.
\end{theorem}

The following theorem is proved in~\cite{NPOM1arxiv}:
\begin{theorem}
\label{thm:convcomb}
Assume $J$ is a countable set, $\alpha_i$ ($i\in I$) are reals, and $(\mathbf B_{i,n})_{n\in\bbbn}$ ($i\in I$) are
sequences of $\sig$-structures such that
	$\alpha_i=\lim_{n\rightarrow\infty}\frac{|\mathbf B_{i,n}|}{|\bigcup_{j\in I}\mathbf B_{j,n}|}\ (\forall i\in I)$,
	$\sum_{i\in I}\alpha_i=1$, and for each $i\in I$, $(B_{i,n})_{n\in\bbbn}$ is ${\rm FO}^{\rm local}$-convergent.
Then $\mathbf A_n=\bigcup_{i\in I}\mathbf B_{i,n}$ is ${\rm FO}^{\rm local}$-convergent.
Also, if there exist  $\sig$-modelings $\mathbf L_i$ ($i\in I$) such that
 for each $i\in I$, $\mathbf B_{i,n}\xrightarrow{{\rm FO}^{\rm local}}\mathbf L_i$,
then $\mathbf A_n\xrightarrow{{\rm FO}^{\rm local}}\coprod_{i\in I}(\mathbf L_i,\alpha_i)$.

Moreover, if the family $\{(\mathbf B_{i,n})_{n\in\bbbn}: i\in I\}$ is uniformly elementarily-convergent,
then $(\mathbf A_n)_{n\in\bbbn}$ is ${\rm FO}$-convergent.
Also, if there exist  $\sig$-modelings $\mathbf L_i$ ($i\in I$) such that for each $i\in I$ it holds
$\mathbf B_{i,n}\xrightarrow{{\rm FO}}\mathbf L_i$ (for $i>0$ and $i=0$ if $\alpha_0>0$)
and
$\mathbf B_{0,n}\xrightarrow{{\rm FO}_0}\mathbf L_0$ (if $\alpha_0=0$)
then
$\mathbf A_n\xrightarrow{{\rm FO}}\coprod_{i\in I}(\mathbf L_i,\alpha_i)$.
\end{theorem}

\section{Limit of Forests}
\label{sec:trees}
\subsection{Limit of Residual Sequences of Forests}
In this section we shall prove that every 
${\rm FO}_1^{\rm local}$-convergent residual sequence of trees
 has a modeling ${\rm FO}_1^{\rm local}$-limit that satisfies
the Strong Finitary Mass Transport Principle.

In this section, we consider rooted forests with edges oriented from the roots. Roots are marked by unary relation $M$ and arcs
by binary relation $R$. Rooted forests are defined by the following countable set of axioms:
\begin{enumerate}
	\item for each $r\in\bbbn$, a formula stating that two distinct roots are at distance at least $r$ (for each $r\in\bbbn$);
\item  a formula stating that every vertex has indegree exactly one if it is not a root, and zero if it is a root;
\item for each $r\in\bbbn$, a formula stating that there is no circuit of length $r$.
\end{enumerate}
In other words, a rooted forest is an directed acyclic graph such that all the vertices but the roots (which are sources) have indegree $1$, and
such that each connected component contains at most one root.
Note that a rooted forest has two types of connected components: connected components that contain a root, and (infinite) connected components
that do not contain a root. 

We first state a lemma relating first-order properties of $p$-tuples in a rooted forest to first-order properties of individual vertices.

\begin{lemma}
\label{lem:npto1}
Fix rooted forests $\mathbf Y,\mathbf Y'$. 
Let $u_1,\dots,u_p$ be $p$ vertices of $\mathbf Y$, let $u_1',\dots,u_p'$
be $p$ vertices of $\mathbf Y'$, and let $r,n\in\bbbn$.
We denote by ${\rm Father}$ both the first-order defined mapping that maps a non-root vertex to its unique in-neighbor and leaves roots fixed.

Assume that for every $1\leq i\leq p$ and every $pr$-local formula $\phi\in{\rm FO}_1^{\rm local}$ with quantifier rank
at most $n+r$ it holds 
$$\mathbf Y\models\phi(u_i)\quad\iff\quad \mathbf Y'\models\phi(u_i')$$ 
and
that for every $1\leq i,j\leq p$ and every $0\leq k,l\leq r$, it holds
$$\mathbf Y\models{\rm Father}^k(u_i)={\rm Father}^l(u_j)\quad\iff\quad\mathbf Y'\models{\rm Father}^k(u_i')={\rm Father}^l(u_j').$$

Then, for every $r$-local formula $\psi\in{\rm FO}_p^{\rm local}$  with quantifier rank
at most $n$ it holds
$$\mathbf Y\models\psi(u_1,\dots,u_p)\quad\iff\quad\mathbf Y'\models\psi(u_1',\dots,u_p').$$
\end{lemma}
\begin{proof}
%In the proof we consider $p+1$ simultaneous Ehrenfeucht-Fra{\"\i}ss{\'e} games.
The proof is similar to the proof of Lemma~4.13 in~\cite{NPOM1arxiv}.
\end{proof}

\begin{lemma}
\label{lem:tree1}
Every ${\rm FO}_1^{\rm local}$-convergent residual sequence of forests
$(Y_n)_{n\in\bbbn}$ has a modeling ${\rm FO}_1^{\rm local}$-limit that satisfies the
Strong Finitary Mass Transport Principle.
\end{lemma}
\begin{proof}
Let $(Y_n)_{n\in\bbbn}$ be an ${\rm FO}_1^{\rm local}$-convergent residual sequence of forests.
We mark a root (by unary relation $M$) in each connected component of $Y_n$ and orient the edges of $Y_n$ from the root
(we denote by $R(x,y)$ the binary relation expressing the existence of an arc from $x$ to $y$). 
By extracting
a subsequence, we assume that (the rooted oriented) $(Y_n)_{n\in\bbbn}$ is ${\rm FO}_1^{\rm local}$-convergent.

Connected components of the limit may or not contain a root.
For instance, if $Y_n$ is the union of $\sqrt{n}$ copies of stars of order $\sqrt{n}$, then every connected 
component in the limit contains a root; however, if $Y_n$ is a path of length $n$, only one connected component (with zero measure) in the limit
contains a root.

Local formulas form a Boolean algebra. Let $S(\mathcal B({\rm FO}_1^{\rm local}))$ be the associated Stone space, and let 
$S_1$ be the closed subspace of $S(\mathcal B({\rm FO}_1^{\rm local}))$ formed by all the $T\in S(\mathcal B({\rm FO}_1^{\rm local}))$ 
that contain all the axioms of rooted forests (see the beginning of this section).
%such that the following (countably many) formulas belong to $T$: 

As $(Y_n)_{n\in\bbbn}$ is ${\rm FO}_1^{\rm local}$-convergent,  there exists (see~\cite{CMUC,NPOM1arxiv}) a limit measure $\mu$ on $S_1$ such that
for every $\phi\in {\rm FO}_1^{\rm local}$ it holds
$$
\lim_{n\rightarrow\infty}\langle\phi,Y_n\rangle=\int_{S_1}\mathbf 1_{K(\phi)}(T)\,{\rm d}\mu(T),
$$
where $K(\phi)=\{T\in S_1: \phi\in T\}$.

We partition $S_1$ into countably many measurable parts as follows:
\begin{itemize}
	\item for each non-negative integer $r$, $S_1^{(r)}$ denotes the clopen subset of $S_1$ defined by
	$$S_1^{(r)}=\{T: ((\exists z)\ M(z)\wedge {\rm dist}(z,x_1)=r)\in T\};$$
	\item $S_1^\circ$ is the closed subset of $S_1$ defined by
	$$S_1^\circ=S_1\setminus\bigcup_{r\geq 0}S_1^{(r)}.$$
\end{itemize}
We further define a measurable mapping $\zeta:[0,1)\rightarrow[0,1)$ as follows:
%Let $Q$ be the set of reals in $[0,1)$ of the form $p/2^q$ (with $p,q\in\bbbn$).
Let $x\in[0,1)$, $x=\sum_{i\geq 0}x_i2^{-i}$ (with $\{i: x_i=1\}$ not cofinite). 
%(the $x_i$ are uniquely defined if $x\notin Q$). 
We define
$$\zeta(x)=(\sum_{i\in\bbbn}x_{2i}2^{-i}) \bmod 1.$$
%\begin{cases}
%(\sum_{i\in\bbbn}x_{2i}2^{-i}) \bmod 1&\text{if }x\notin Q\\
%0&\text{otherwise}
%\end{cases}$$
%(Note that for every $x\in[0,1)$, $\zeta^{-1}(x)$ in infinite.)

We define $F:S_1\rightarrow S_1$ by
$$
F(T)=\begin{cases}
T&\text{if }M(x_1)\in T\\
\{\phi:\ ((\exists z) R(z,x_1)\wedge\phi(z))\in T\}&\text{otherwise}
\end{cases}
$$
(Note that $F(T)$ is indeed an ultrafilter of $\mathcal B({\rm FO}_1^{\rm local})$ 
as there exists exactly one $z$ such that $R(z,x_1)$.)

Define $w:S_1\setminus S_1^{(0)}\rightarrow\{0,1,\dots,\infty\}$ as the supremum of the integers
$k$ such that there exists a tree $\mathbf A$ with universe $A$ and $a\in A$ such that
there are at least $k$ sons $b_1,\dots,b_k$ of $a$ such that 
	for all $\phi\in{\rm FO}_1^{\rm local}$ it holds $\mathbf A\models\phi(b_i)$ if and only if $\phi\in T$.

Finally, we define the mapping $\xi:(S_1\setminus S_1^{(0)})\times[0,1)\rightarrow [0,1)$ by
$$
\xi(T,\alpha)=\begin{cases}
	w(T)\alpha\bmod 1&\text{if }w(T)<\infty\\
	\zeta(\alpha)&\text{otherwise}
	\end{cases}
$$
Note that for every $(T,\alpha)\in (S_1\setminus S_1^{(0)})\times[0,1)$, the set
$$\{\alpha'\in[0,1):\ \xi(T,\alpha')=\xi(T,\alpha)\}$$ 
has cardinality $w(T)$ (if $w(T)<\infty$) and is infinite (if $w(T)=\infty$).

Using these special functions, we can construct limit modelings of residual sequences of forests with the following simple form:

Let $Z=\bigl(\bigcup_{r\geq 0}S_1^{(r)}\times [0,1)\bigr)\cup \bigl(S_1^\circ\times [0,1)\times \mathbb S^1\bigr),$
where $\mathbb S^1$ is the unit circle (identified here with reals mod $2\pi$).
It is clear that $Z$ is the standard Borel space.
We fix a real $\theta_0$ such that $\theta_0$ and $\pi$ are incommensurable, and we define a rooted directed forest $\mathbf Z$ on $Z$ has follows:
\begin{itemize}
	\item for $z\in Z$, it holds $M(z)$ if and only if $z\in S_1^{(0)}\times [0,1)$;
	\item for positive integer $r$ and $z\in S_1^{(r)}\times [0,1)$, $z=(T,\alpha)$, the vertex $z$ has exactly 
	one incoming edge from the vertex $(F(T),\xi(T,\alpha))\in S_1^{(r-1)}\times[0,1)$;
	\item for $z\in S_1^\circ\times[0,1)\times\mathbb S^1$, $z=(T,\alpha,\theta)$ the vertex $z$ has exactly
	one incoming edge from the vertex $(F(T),\xi(T,\alpha),\theta+\theta_0)$.
\end{itemize}
It is easily checked that for $z\in Z$ ($z=(T,\alpha)$ or $z=(T,\alpha,\theta)$), the set of formulas
$\phi\in{\rm FO}_1^{\rm local}$ such that $\mathbf Z\models\phi(z)$ is exactly $T$.

%%%%%%%%%%%%%%%%%%%%
We now prove that $\mathbf Z$ is a relational sample space.
%\begin{lemma}
%\label{thm:zrss}
%The rooted forest $\mathbb Z$ is a relational sample space.
%\end{lemma} 
%\begin{proof}
It suffices to prove that for every $p\in\bbbn$ and every $\varphi\in{\rm FO}_p^{\rm local}$ the set 
$$
\Omega_\varphi(\mathbf Z)=\{(v_1,\dots,v_p)\in V_h^p:\ \mathbf
Z\models\varphi(v_1,\dots,v_p)\}
$$
is measurable.

Let $\varphi\in{\rm FO}_{p}^{\rm local}$ and let $n={\rm qrank}(\varphi)$. 
We partition $V_h$ into equivalence classes modulo
$\equiv^{n+r}$, which we denote $C_1,\dots,C_N$.
Let $i_1,\dots,i_p\in[N]$ and, for $1\leq j\leq p$, let
$v_{j}$ and $v_{j}'$ belong to $C_{i_j}$.
According to
Lemma~\ref{lem:npto1}, if for every $1\leq i<j\leq p$ and $1\leq k,l\leq r$ it holds
$${\rm Father}^k(u_i)={\rm Father}^{l}(u,j)\quad\iff\quad {\rm Father}^k(u_i')={\rm Father}^{l}(u,j')$$
then it holds
$$(v_1,\dots,v_p)\in \Omega_\varphi(\mathbf Z)\quad\iff\quad (v_1',\dots,v_p')\in\Omega_\varphi(\mathbf Z).$$
According to the encoding of the vertices of $\mathbf Z$, the conditions
on the common ancestors rewrite as equalities and inequalities of iterated measurable functions $[0,1]\rightarrow[0,1)$.
It follows that $\Omega_\varphi(\mathbf Z)$ is measurable. Thus $\mathbf Z$ is a relational sample space.
%\end{proof}

We define a probability measure $\nu_{\mathbf Z}$ on $\mathbf Z$ to turn $\mathbf Z$ into a modeling as follows:
\begin{itemize}
	\item on $\bigcup_{r\geq 0}S_1^{(r)}\times [0,1)$, $\nu_{\mathbf Z}$ is the product of the restriction of $\mu$ and the Borel measure on $[0,1)$;
	\item on $S_1^\circ\times [0,1)\times\mathbb S^1$, $\nu_{\mathbf Z}$ is the product of the restriction of $\mu$, the Borel measure on $[0,1)$, and the Haar
	(rotation invariant) probability measure of $\mathbb S^1$.
\end{itemize}

It is easily checked from the above definition of $\mathbf Z$ (and of course $\zeta,F,\xi,w$) that 
the modeling $\mathbf Z$ is a modeling ${\rm FO}_1^{\rm local}$-limit of $(Y_n)_{n\in\bbbn}$, and that 
if $\mu$ satisfies the Finitary Mass Transport Principle then $\mathbf Z$ satisfies
the Strong Finitary Mass Transport Principle.
\end{proof}

The construction of a modeling ${\rm FO}_1^{\rm local}$-limit for 
the root-free part resembles spinning wheel of a limit forest (cf~\cite{spw}) and it 
 is schematically illustrated on Fig~\ref{fig:band}.

\begin{figure}[ht]
	\centering
		\includegraphics[width=0.75\textwidth]{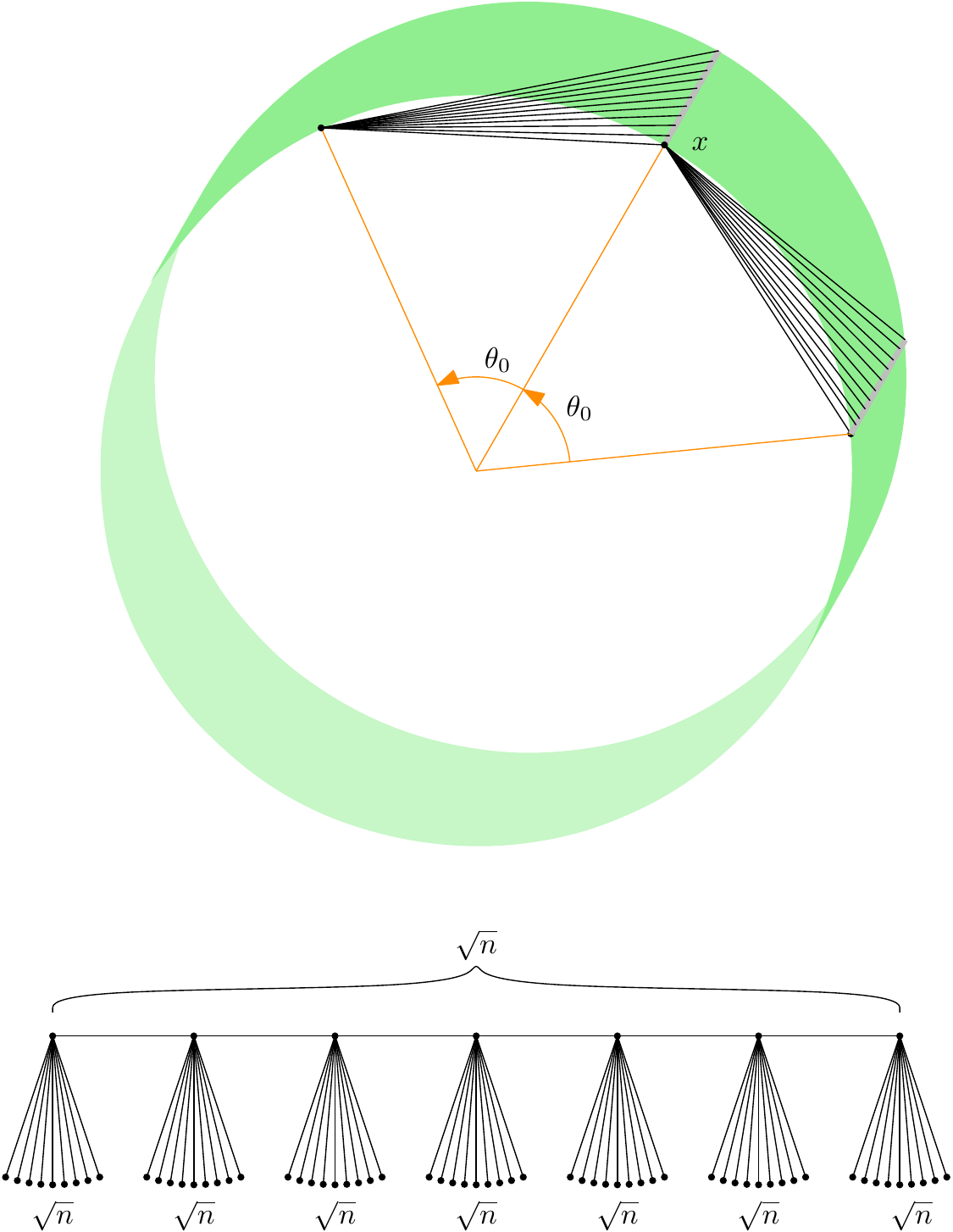}
	\caption{Modeling ${\rm FO}_1^{\rm local}$-limit of a $\sqrt{n}$-path of stars of order $\sqrt{n}$. Vertex $x$ is adjacent to $2^{\aleph_0}$ vertices on a segment and to two vertices obtained by rotation of $\pm\theta_0$, where $\theta_0$ is irrational to $\pi$. 
	%Note that one segment has been removed from the cylinder (to get the two end-vertices of the path).
	}
	\label{fig:band}
\end{figure}

\subsection{Limit of $\rho$-non-dispersive Sequences of Trees}
Let $\sig$ be the signature of graphs, $\sig^\bullet$ the signature of graphs with additional unary relation $R$, 
 $\sig^+$ the signature of graphs with additional unary relations $R$ and $P$, $\sig^\omega$ the signature
of graphs with countably many additional unary relations $M_i$ and $N_i$ ($i\in\bbbn$). 
We consider two basic interpretation schemes, which we made use of already in \cite{NPOM1arxiv}:
\begin{enumerate}
  \item  $\mathsf I_{Y\rightarrow F}$ is a basic interpretation scheme 
  of $\sig^+$-structures in $\sig^\bullet$-structures defined as follows:
	for every $\sig$-structure $\mathbf A$, the domain of $\mathsf I_{Y\rightarrow F}(\mathbf A)$
	is the same as the domain of $\mathbf A$, and it holds (for every $x,y\in A$):
\begin{align*}
\mathsf I_{Y\rightarrow F}(\mathbf A)\models x\sim y&\quad\iff\quad\mathbf A\models (x\sim y)\wedge\neg R(x)\wedge\neg R(y)\\
\mathsf I_{Y\rightarrow F}(\mathbf A)\models R(x)&\quad\iff\quad\mathbf A\models (\exists z)\ R(z)\wedge (z\sim x)\\
\mathsf I_{Y\rightarrow F}(\mathbf A)\models P(x)&\quad\iff\quad\mathbf A\models R(x)
\end{align*}
In particular, 	if $Y$ is a $\sig^\bullet$-tree, with a single vertex marked by $R$ (the root),
 $\mathsf I_{Y\rightarrow F}$ maps $Y$ into a $\sig^+$-forest
$\mathsf I_{Y\rightarrow F}(Y)$, formed by
 the subtrees rooted at the sons of the former root (roots marked by $R$)
and a single vertex rooted principal component
(the former root, marked $P$);
\item $\mathsf I_{F\rightarrow Y}$ is a basic interpretation scheme 
  of $\sig^\bullet$-structures in $\sig^+$-structures defined as follows:
	for every $\sig^+$-structure $\mathbf A$, the domain of $\mathsf I_{F\rightarrow Y}(\mathbf A)$
	is the same as the domain of $\mathbf A$, and it holds (for every $x,y\in A$):
\begin{align*}
\mathsf I_{F\rightarrow A}(\mathbf A)\models x\sim y&\quad\iff\quad\mathbf A\models (x\sim y)\vee R(x)\wedge P(y)\vee R(y)\wedge P(x)\\
\mathsf I_{F\rightarrow A}(\mathbf A)\models R(x)&\quad\iff\quad\mathbf A\models P(x)
\end{align*}
In particular, 	$\mathsf I_{F\rightarrow Y}$ maps
a $\sig^+$-forest $F$ with all connected components rooted by $R$, except exactly one rooted by $P$
into a $\sig^\bullet$-tree
$\mathsf I_{F\rightarrow Y}(F)$ by making each non principal root a son of the principal root;
\end{enumerate}

\begin{lemma}
\label{lem:tree2}
Every  ${\rm FO}$-convergent $\rho$-non-dispersive sequence of rooted trees
$(Y_n)_{n\in\bbbn}$ has a modeling ${\rm FO}$-limit.
\end{lemma}
\begin{proof}
Let $(Y_n)_{n\in\bbbn}$ be an ${\rm FO}$-convergent $\rho$-non-dispersive 
sequence of rooted trees.
Then ${\mathsf I}_{Y\rightarrow F}(Y_n)_{n\in\bbbn}$ is an ${\rm FO}$-convergent sequence of rooted
forests. According to Theorem~\ref{thm:ecomb}, there exist $I\subseteq\bbbn$, reals $\alpha_i$,
sequences $(B_{i,n})_{n\in\bbbn}$ for $i\in I\cup\{0\}$, such that:
\begin{itemize}
	\item $\alpha_0\geq 0$, $\alpha_i>0$ (for $i\in I$), and $\sum_{i\in I\cup\{0\}}\alpha_i=1$;
	\item $Y_n$ is the union of the $B_{i,n}$ ($i\in I\cup\{0\}$);
	\item $\lim_{n\rightarrow\infty}|B_{i,n}|/|Y_n|=\alpha_i$ (for $i\in I\cup\{0\}$);
	\item $(B_{0,n})_{n\in\bbbn}$ is an ${\rm FO}^{\rm local}$-convergent residual sequence if $\alpha_0>0$;
	\item $(B_{i,n})_{n\in\bbbn}$ is an ${\rm FO}^{\rm local}$-convergent non-dispersive sequence;
	\item the family $\{(B_{i,n},\rho_{i,n})_{n\in\bbbn}: i\in I\}$ is uniformly elementarily convergent.
\end{itemize}
In this situation we apply Theorem~\ref{thm:convcomb}:

We denote by $\mathbf L_i$ and $\mathbf L_0$ the modeling limits, so that
\begin{itemize}
	\item   $(B_{i,n})\xrightarrow{{\rm FO}}\mathbf L_i$ (for $i\in\bbbn$);
\item $(B_{0,n})\xrightarrow{{\rm FO}}\mathbf L_0$ (if $\alpha_0>0$), and 
$(B_{0,n})\xrightarrow{{\rm FO}_0}\mathbf L_0$ (otherwise).
\end{itemize}
Then we have (by Theorem~\ref{thm:convcomb}):
$$Y_n\xrightarrow{{\rm FO}}\mathsf I_{F\rightarrow Y}\bigl(\coprod_{i\in I\cup\{0\}}(\mathbf L_i,\alpha_i)\bigr).$$

It is easily checked that each sequence $(B_{i,n})_{n\in\bbbn}$ is $\rho$-non-dispersive
(for $B_{i,n}$ rooted at its marked vertex), as a direct consequence
of the fact that $(Y_n)_{n\in\bbbn}$ (rooted at marked vertex) is $\rho$-non-dispersive.

If we repeat the same process on each $\rho$-non-dispersive sequence $(B_{i,n})_{n\in\bbbn}$ (for $i\in I\setminus\{0\}$),
 we inductively construct a countable rooted
tree $S$ and, associated to each node $v$ of the tree, a residual sequence of forests $(F_{v,n})_{n\in\bbbn}$ and
a weight $\lambda_v$. 
If we have started with a $\rho$-non-dispersive sequence, then for every $\epsilon>0$ there is integer $d$ such that
the sum of the measures of the residues attached to the nodes at height at most $d$ is at least $1-\epsilon$. 

According to Lemma~\ref{lem:tree1}, for each residual ${\rm FO}$-convergent sequence $(F_{v,n})_{n\in\bbbn}$ of forests
there is a rooted tree modeling $\mathbf L^0_v$ that is the ${\rm FO}_1^{\rm local}$-limit of $({\mathsf I}_{F\rightarrow Y}(F_{v,n}))_{n\in\bbbn}$. Hence, according to Lemmas~\ref{lem:mkres},~\ref{lem:mkclean}, and~\ref{lem:mkfull}, there is
a rooted tree modeling $\mathbf L_v$, which is the ${\rm FO}$-limit of $({\mathsf I}_{F\rightarrow Y}(F_{v,n}))_{n\in\bbbn}$.

The grafting
of the modelings $\mathbf L_v$ on the rooted tree $S$ (with weights $\lambda_v$) form a final modeling $\mathbf L$.

We prove that $\mathbf L$ is a relational sample space:
 each first-order
definable subset of $L^p$ is a $\mathcal L_{\omega_1\omega}$-definable subsets of the countable union of all the $\mathbf L_v$ in
which the roots of all the roots have been marked by distinct unary relations $M_v$. As the used language in countable, it follows
from Lemma~\ref{lem:omega} that $\mathcal L_{\omega_1\omega}$-definable subsets are Borel measurable. 

Let $d\in\bbbn$, and let $Y_n^{(d)}$ be the subtree of $Y_n$ induced by vertices at distance at most $d$ from the
root. As the trees $Y_n^{(d)}$ are obtained by an obvious interpretation, we get that $(Y_n^{(d)})_{n\in\bbbn}$ is ${\rm FO}$-convergent. Now consider $\mathbf L^{(d)}$, obtained from $\mathbf L$ by restricting to the set $X_d$ of the vertices at distance at most $d$ from the root. As the set $X_d$ is first-order defined, it is measurable. It follows that $X_d$ (with induced
$\sigma$-algebra) is a standard Borel space, and that $\mathbf L^{(d)}$ is a relational sample space.
We define a probability measure on $\mathbf L^{(d)}$ by $\nu_{\mathbf L^{(d)}}=\nu_{\mathbf L}/\nu_{\mathbf L}(X_d)$, thus
defining the modeling $\mathbf L^{(d)}$. By applying iteratively (at depth $d$) Theorem~\ref{thm:convcomb} and the interpretation
$\mathsf I_{F\rightarrow Y}$ we easily deduce that $\mathbf L^{(d)}$ is a modeling ${\rm FO}$-limit of $(Y_n^{(d)})_{n\in\bbbn}$.
Thus, according to Lemma~\ref{lem:boundedh}, we deduce that $\mathbf L$ is a modeling ${\rm FO}^{\rm local}$-limit of the sequence $(Y_n)_{n\in\bbbn}$.
By Lemma~\ref{lem:mkfullc}, we deduce that $(Y_n)_{n\in\bbbn}$ has a modeling ${\rm FO}$-limit, which is the union
of $\mathbf L$ and a countable graph.
\end{proof}

\begin{maintheorem}{2}
\label{thm:mtree}
Every ${\rm FO}$-convergent sequence of finite forests has a modeling ${\rm FO}$-limit  that satisfies the Strong Finitary Mass Transport Principle.
\end{maintheorem}
\begin{proof}
The theorem is an immediate consequence of Theorem~\ref{thm:main} and Lemmas~\ref{lem:tree1} and~\ref{lem:tree2}.
\end{proof}

\section{Conclusion}
Compared to the results obtained in~\cite{NPOM1arxiv} for rooted trees with bounded height, 
we do not have inverse theorem for tree-modeling. Our conjecture is that a tree-modeling that satisfy 
	the Strong Finitary Mass Transport Principle and whose complete theory has the finite model property is the modeling limit
	of a sequence of finite trees.

We believe that our approach can be used to obtain modeling limits of further classes of graphs.
In particular, we believe that the structure ``residual limits grafted on a countable skeleton'' might well
be universal for sequences of nowhere dense graphs.

%\bibliographystyle{amsplain}
%\bibliography{biblio,bib}

\providecommand{\noopsort}[1]{}\providecommand{\noopsort}[1]{}
\providecommand{\bysame}{\leavevmode\hbox to3em{\hrulefill}\thinspace}
\providecommand{\MR}{\relax\ifhmode\unskip\space\fi MR }
% \MRhref is called by the amsart/book/proc definition of \MR.
\providecommand{\MRhref}[2]{%
  \href{http://www.ams.org/mathscinet-getitem?mr=#1}{#2}
}
\providecommand{\href}[2]{#2}

\end{document}